\documentclass[12pt,reqno]{amsart}

\usepackage{amsfonts}
\usepackage{amsmath}
\usepackage{amssymb}
\usepackage{amsthm}
\usepackage{bm}
\usepackage{enumerate} 
\usepackage{hyperref}
\usepackage{cite}
\usepackage[margin=1.2in]{geometry}
\usepackage{mathrsfs} 
\usepackage{graphicx} 
\usepackage{tikz-cd}

\usepackage{xcolor}

\theoremstyle{plain}
\newtheorem{theorem}{Theorem}[section]
\newtheorem{lemma}[theorem]{Lemma}
\newtheorem{proposition}[theorem]{Proposition}
\newtheorem{corollary}[theorem]{Corollary}

\theoremstyle{definition}
\newtheorem{definition}[theorem]{Definition}

\newtheorem{remark}[theorem]{Remark}

\newtheorem{example}[theorem]{Example}

\numberwithin{equation}{section}

\newcommand\N{\mathbb{N}}

\newcommand\R{\mathbb{R}}

\newcommand\C{\mathbb{C}}

\newcommand\B{\mathfrak{B}}

\newcommand\ev[2]{\langle#1,#2\rangle}

\DeclareMathOperator\PL{PL}
\DeclareMathOperator\GI{QI}
\DeclareMathOperator\GL{GL}

\newcommand\ModK[2]{\mathfrak{m}_{#1}(#2)}
\newcommand\Mod[1]{\ModK{}{#1}}
\newcommand\ModOm{\Mod{\omega}}
\newcommand\ModKOm[1]{\ModK{#1}{\omega}}
\newcommand\ModKapOm{\ModKOm{\kappa}}

\begin{document}

\title[Quantitative uncertainty principles]{Quantitative uncertainty principles for time-frequency Gaussian decay}

\author[L. Neyt]{Lenny Neyt}
\address{University of Vienna\\ Faculty of Mathematics\\ Oskar-Morgenstern-Platz 1 \\ 1090 Wien\\ Austria}
\thanks{The research of L. Neyt was funded in whole by the Austrian Science Fund (FWF) 10.55776/ESP8128624. For open access purposes, the author has applied a CC BY public copyright license to any author-accepted manuscript version arising from this submission.}
\email{lenny.neyt@univie.ac.at}

\author[J. Toft]{Joachim Toft}
\address{Department of Mathematics \\ Linn\ae us University \\  Växjö \\ Sweden}
\email{joachim.toft@lnu.se}

\author[J. Vindas]{Jasson Vindas}
\thanks{The work of J. Vindas was supported by Ghent University through the grant bof/baf/4y/2024/01/155.}
\address{Department of Mathematics: Analysis, Logic and Discrete Mathematics \\ Ghent University \\ Krijgslaan 281 \\ 9000 Gent \\ Belgium}
\email{jasson.vindas@UGent.be}


\begin{abstract}
For real symmetric positive definite matrices $A$ and $B$, we characterize when a function $f \in L^2(\R^d)$ satisfies
	\[ |f(x)| \lesssim e^{-(\frac12 - \lambda) \ev{Ax}{x}} \quad \text{and} \quad |\widehat{f}(\xi)| \lesssim e^{-(\frac12 - \lambda) \ev{B \xi}{\xi}} , \qquad \forall \lambda > 0 , \]
or even more specified time-frequency decay estimates, in terms of the skewed Hermite series expansion of $f$.
We also consider coordinate-wise time-frequency decay and determine when it becomes equivalent to the same bounds on the skewed Hermite coefficients.
\end{abstract}

\maketitle

\section{Introduction}

At the heart of the uncertainty principle in harmonic analysis lies the fact that 
a function $f \in L^2(\R^d)$ and its Fourier transform
	\[ \widehat{f}(\xi) = (2 \pi)^{-\frac{d}{2}} \int_{\R^d} f(x) e^{- i \ev{x}{\xi}} dx \]
cannot simultaneously be too well-localized. 
The mathematical formulation of this fascinating principle can take various forms, see for instance  \cite{FS1997, H-J-UncertaintyPrinciple, Jaming2006,T-IntroUncertaintyPrinciple} and the references therein for extensive surveys.
One of the first and most influential results in this spirit was given by Hardy \cite{Hardy33}, based on a remark made by Wiener. 
His result states that if $f \in L^2(\R)$ is non-zero and satisfies
	\begin{equation}
	 	\label{eq:GSTFDecay}
		|f(x)| \lesssim e^{-\frac{\eta}{2} |x|^2} \quad \text{ and } \quad |\widehat{f}(\xi)| \lesssim e^{- \frac{\eta}{2} |\xi|^2}  , 
	\end{equation}
for some $\eta > 0$, then necessarily $\eta \leq 1$.
In the critical case $\eta = 1$, it must even hold that $f(x) = c e^{-\frac12 |x|^2}$
for some $c \in \C$,  that is, $f$ is a Fourier-invariant Gaussian.

Hardy further showed that if
	\begin{equation}
		\label{eq:PolynomialDecay} 
		|f(x)| \lesssim (1 + |x|)^N e^{-\frac12 |x|^2} \quad \text{ and }
		\quad |\widehat{f}(\xi)| \lesssim (1 + |\xi|)^N e^{-\frac12 |\xi|^2} , 
	\end{equation}
for some $N \in \N$, then $f$ must necessarily be of the form
$$
f(x)=p(x) e^{-\frac12 |x|^2} ,
$$
for a certain polynomial $p$ of degree at most $N$.

Inspired by H\"ormander's proof  \cite{H-UniqueThmBeurlingFourier} of Beurling's uncertainty principle for Fourier transforms, Bonami, Damange, and Jaming established \cite[Proposition 3.4]{BDJ2003} the following important multidimensional generalization of Hardy's results. 
Among other aspects, their result stresses the fact that merely coordinate-wise time-frequency bounds already suffice as hypotheses for this uncertainty principle.

	\begin{theorem}
		\label{t:LogCase}
		Let $a, b > 0$.
		Let $f \in L^2(\R^d)$ be such that, for some $N \in \N$ and any $j \in \{1, \ldots, d\}$,
			\begin{equation} 
				\label{eq:PolynomialDecayCoordinatesIntro} 
				|f(x)| \lesssim (1 + |x_j|)^N e^{-\frac{a}{2} |x_j|^2} \quad \text{and} \quad |\widehat{f}(\xi)| \lesssim (1 + |\xi_j|)^N e^{-\frac{b}{2} |\xi_j|^2} .
			\end{equation}
		Then the following is true:
			\begin{itemize}
				\item[$(i)$] If $ab > 1$, then $f=0$.
				\item[$(ii)$] If $ab = 1$, then $f(x) = p(x) e^{-\frac{a}{2} |x|^2}$ for some polynomial $p$ of degree at most $N$ in each variable.
				\item[$(iii)$] If $ab<1$, then the set of all $f \in L^2(\R^d)$ satisfying \eqref{eq:PolynomialDecayCoordinatesIntro} forms a dense subspace of $L^2(\R^d)$.
			\end{itemize}
	\end{theorem}
	
The Hardy uncertainty principle and its quantitative versions are in particular relevant for the study of Schr\"odinger equations; see, for instance, \cite{C-F-SharpHardyUncertaintyGaussianProfSchrodinger, C-G-M-HardyUncertaintySchrodingerHamiltonians, E-K-P-V-SharpHardyUncertaintySchrodinger, FB-M-DynamicalHardyUncertainty} for recent developments in this direction.
Variants of Theorem \ref{t:LogCase} for other classes of time-frequency operators have likewise attracted considerable attention; see, for example, \cite{BDJ2003, G-S-MoreUncertainty, G-Z-HardyThmSTFT}.

The goal of this work is to determine the structure imposed on any $f \in L^2(\R^d)$ by time-frequency decay bounds that intermediately lie between rescaled versions of \eqref{eq:PolynomialDecay} and \eqref{eq:GSTFDecay}. In particular, our results shall significantly extend Theorem \ref{t:LogCase}.  Let $A \in \GL_+(\R^d)$, where $\GL_+(\R^d)$ stands for the space of real symmetric positive definite $d \times d$ matrices.
We will be interested in functions $f \in L^2(\R^d)$ that satisfy
	\begin{equation}
		\label{eq:NearlyOptimalTFDecay}
		|f(x)| \lesssim e^{-(\frac12 - \lambda) \ev{Ax}{x}} \quad \text{ and } \quad |\widehat{f}(\xi)| \lesssim e^{-(\frac12 - \lambda) \ev{A^{-1} \xi}{\xi}} ,
		 \qquad \text{for all } \lambda > 0 ,
	\end{equation}
or even some more refined time-frequency estimates.
Moreover, we also consider coordinate-wise variants of the bounds. 
An effective way of studying such functions, and the main aim of this article, is by looking at the skewed Hermite series expansion of $f$, as we now proceed to explain.

Let
	\begin{equation} 
		h_\alpha(x) = (-1)^{|\alpha|} \pi^{-\frac{d}{4}} (2^{|\alpha|} \alpha!)^{-\frac12} e^{\frac12 |x|^2} \frac{\partial^{\alpha}}{\partial x^\alpha} [e^{-|x|^2}] , \qquad \alpha \in \N^d , 
	\end{equation}
be the classical $d$-dimensional Hermite functions. 
The \emph{$A$-skewed Hermite functions} are given by
	\begin{equation} 
		h_{A, \alpha}(x) := (\det A)^{\frac14} h_\alpha(A^{1/2} x) . 
	\end{equation}
It holds that $\widehat{h}_{A, \alpha} = (-i)^{\alpha} h_{A^{-1}, \alpha}$. 
The set $\{ h_{A, \alpha} \mid \alpha \in \N^d \}$ forms an orthonormal basis for $L^2(\R^d)$, so that any $f \in L^2(\R^d)$ takes the form
	\[ f = \sum_{\alpha \in \N^d} H_A(f, \alpha) h_{A, \alpha} , \quad \text{with } H_A(f, \alpha) = (f, h_{A, \alpha})_{L^2} . \]
We immediately get
	\[ |H_A(f, \alpha)| = |H_{A^{-1}}(\widehat{f}, \alpha)| \quad \text{and} \quad \sum_{\alpha \in \N^d} |H_A(f, \alpha)|^2 < \infty . \]
Imposing stronger bounds on the skewed Hermite coefficients $H_A(f, \alpha)$ leads to even faster time-frequency decay. 
For instance, Schwartz \cite{S-ThDist} provided the following characterization of the space of rapidly decreasing smooth functions:
$$
f\in \mathcal{S}(\R^d)
\quad \Longleftrightarrow \quad
|H_A(f, \alpha)| \lesssim |\alpha|^{-r} ,
\qquad \text{for any $r > 0$.}
$$
Later, De Bruijn \cite{DB-ThGenFuncApplWignerWeyl} and Zhang \cite{Zhang63} independetly discovered (in view of \cite{CCK96}) that \eqref{eq:GSTFDecay} holds for some $\eta>0$ if and only if $|H_A(f, \alpha)| \lesssim e^{-r |\alpha|}$ for some $r > 0$.
We refer to \cite{Langenbruch2006, N-T-V-HermExpSpFuncNearlyOptimalTFDecay,Pilipovic88, T-ImagFuncDistSpBargmannTrans} for further Hermite expansion descriptions of broader classes of function spaces.

Our first objective is to characterize \eqref{eq:NearlyOptimalTFDecay} in terms of the skewed Hermite coefficients of $f$. In our work \cite{N-T-V-HermExpSpFuncNearlyOptimalTFDecay}, this sought description was fully achieved for dimension one.
However, for the multidimensional case, our techniques needed to assume uniform decay rates for \emph{every} partial Fourier transform of $f$ (see \cite[Theorem 5.1]{N-T-V-HermExpSpFuncNearlyOptimalTFDecay}) and a complete multivariate characterization avoiding extra partial Fourier transform information appears to be out of reach of the methods from \cite{N-T-V-HermExpSpFuncNearlyOptimalTFDecay}.
In the present text, we develop several new ideas and, in particular, show that, in fact, \eqref{eq:NearlyOptimalTFDecay} suffices to get the desired bounds on the skewed Hermite coefficients.
Also, we will see that even coordinate-wise estimates are enough.
Concretely, we will establish the next result.
Here 
$\{e_j \mid 1 \leq j \leq d\}$ denotes the standard Euclidean basis.

\begin{theorem}
\label{t:CharNearlyOptimalTFDecay}
Let $A \in \GL_+(\R^d)$, and let $v_j = A^{-1/2} e_j$ and $w_j = A^{1/2} e_j$ for $j \in \{1, \ldots, d\}$.
For any $f \in L^2(\R^d)$, the following statements are equivalent:
	\begin{itemize}
		\item[$(i)$] For any $j \in \{1, \ldots, d\}$,
			\begin{equation}
				\label{eq:NearlyOptimalTFDecayCoordinates}
				|f(x)| \lesssim e^{-(\frac{1}{2} - \lambda) |\ev{A x}{v_j}|^2} \quad \text{ and } \quad |\widehat{f}(\xi)| \lesssim e^{-(\frac{1}{2} - \lambda) |\ev{A^{-1} \xi}{w_j}|^2} , \qquad \text{for all } \lambda > 0 .
			\end{equation}
		\item[$(ii)$] $f$ satisifies \eqref{eq:NearlyOptimalTFDecay}.
		\item[$(iii)$] One has
			\begin{equation}
				\label{eq:NearlyOptimalTFDecayHermite}
				|H_A(f, \alpha)| \lesssim e^{-r |\alpha|} , \qquad \text{for all } r > 0 .
			\end{equation}
	\end{itemize}
\end{theorem}

We remark that for the previous theorem, as well as for the next one, the matrix $A^{-1}$ appearing in the frequency decay of \eqref{eq:NearlyOptimalTFDecay} or \eqref{eq:NearlyOptimalTFDecayCoordinates} is vital.
Otherwise, as we will establish in Section \ref{sec:PreciseEst}, if
\[
		|f(x)| \lesssim e^{- (\frac12 - \lambda) \ev{Ax}{x}}  \quad \mbox{and} \quad |\widehat{f}(\xi)| \lesssim e^{- (\frac12 - \lambda) \ev{B \xi}{\xi}} 
\]
hold with another real symmetric positive definite matrix $B\neq A^{-1}$, then the function $f$ is either necessarily null or no quantitative result as in \eqref{eq:NearlyOptimalTFDecayHermite} can hold for it, see Theorem \ref{t:AandBNotInverses} and Remark \ref{remarknoquantitativeuncertainty}.

We shall also obtain a number of refinements of Theorem \ref{t:CharNearlyOptimalTFDecay},
where skewed Hermite expansions are characterized by time-frequency decay estimates governed through suitable weight functions. 
In particular, we deduce the following multidimensional analog of \cite[Theorem 1.3]{N-T-V-HermExpSpFuncNearlyOptimalTFDecay}.
By a weight function, we mean an unbounded nondecreasing function $\omega : [0, \infty) \to [0, \infty)$.
We refer to Section \ref{sec:WeightFunc} for further explanations on weight functions and the notions appearing in the next statement.

\begin{theorem}
\label{t:CharWeightedTFDecay}
Let $A \in \GL_+(\R^d)$. 
Let $\omega$ be a weight function such that $\varphi(t) = \omega(e^t)$ is convex and 
	\begin{equation}
		\label{eq:beta*2}
		\int_x^\infty \frac{\omega(t)}{t^3} dt = O\left(\frac{\omega(x)}{x^2}\right) , \qquad x \to \infty ;
	\end{equation}
let $\varphi^*$ be the Young conjugate of $\varphi$.
A function $f \in L^2(\R^d)$ satisfies
	\begin{equation}
		\label{eq:NearlyOptimalTFDecaySpecified}
		|f(x)| \lesssim e^{-\frac12 \ev{Ax}{x} + \lambda \omega(|x|)} \quad \text{ and } \quad |\widehat{f}(\xi)| \lesssim e^{- \frac12 \ev{A^{-1} \xi}{\xi} + \lambda \omega(|\xi|)} , 
	\end{equation}
for some $\lambda > 0$ (for all $\lambda > 0$) if and only if its skewed Hermite coefficients satisfy
	\begin{equation}
		\label{eq:NearlyOptimalTFDecaySpecifiedHermite}
		|H_A(f, \alpha)| \lesssim \sqrt{\alpha!} e^{-\frac{1}{r} \varphi^*(r |\alpha|)} ,
	\end{equation}
for some $r > 0$ (for all $r > 0$).
\end{theorem}

For general weight functions, coordinate-wise analogs of \eqref{eq:NearlyOptimalTFDecaySpecified} do not suffice to conclude \eqref{eq:NearlyOptimalTFDecaySpecifiedHermite}. 
We will study the latter problem in detail.
In Section \ref{sec:Coordinates}, we introduce the property of \emph{quadratic interpolation (of dimension $d$)}. 
It turns out that this property completely determines whether or not \eqref{eq:NearlyOptimalTFDecaySpecified} may be interchanged with coordinate-wise estimates in Theorem \ref{t:CharWeightedTFDecay}, see Theorems \ref{t:CoordinateWiseSucces} and \ref{t:CounterExample}.
Quadratic interpolation is satisfied when the weight function increases relatively slowly, e.g., when $\omega(t^d) = O(\omega(t))$ (cfr. Lemma \ref{l:SuffCondGI}).
In particular, this is true when $\omega(t) = (\log_+ t)^{1+a}$, for some $a \geq 0$, see Theorem \ref{t:CoordinatesLog}.
The special case $a = 0$ then recovers Theorem \ref{t:LogCase}.
Notable examples of weight functions not satisfying quadratic interpolation of any dimension are $\omega(t) = t^a$ for $0 < a < 2$, see Theorem \ref{t:Countexamplet^a}.

Let us say a few words about the proofs of Theorems \ref{t:CharNearlyOptimalTFDecay} and \ref{t:CharWeightedTFDecay}.
Using the (skewed) Bargmann transform \cite{B-HilbertSpAnalFuncAssocIntTrans}, most of our work is done using entire functions.
In dimension one, both results were already shown by the authors in \cite{N-T-V-HermExpSpFuncNearlyOptimalTFDecay}, though here we will give more precise quantified estimates in Section \ref{sec:PreciseEst}.
In fact, even for higher dimensions, the proof of the sufficiency of \eqref{eq:NearlyOptimalTFDecayHermite}, respectively \eqref{eq:NearlyOptimalTFDecaySpecifiedHermite}, to guarantee \eqref{eq:NearlyOptimalTFDecay}, respectively \eqref{eq:NearlyOptimalTFDecaySpecified}, is contained there, see \cite[Theorems 5.1 and 5.2]{N-T-V-HermExpSpFuncNearlyOptimalTFDecay}.
However, it should be emphasized that this is not the main challenge in the analysis of the problem.
As already was the case in \cite{N-T-V-HermExpSpFuncNearlyOptimalTFDecay}, the primary obstacle happens to be the converse direction, that is, showing that \eqref{eq:NearlyOptimalTFDecay}, respectively \eqref{eq:NearlyOptimalTFDecaySpecified}, implies \eqref{eq:NearlyOptimalTFDecayHermite}, respectively \eqref{eq:NearlyOptimalTFDecaySpecifiedHermite}.
In order to achieve proofs of these implications, in dimension at least two, we have developed a novel approach.
One of the key ideas lies in the following proposition, which reduces establishing certain growth bounds for entire functions to proving directional versions of such estimates. 
Here $\mathbb{S}^{d-1}$ is the unit $d$-sphere in $\R^d$.
\begin{proposition}
\label{t:Homogeneous}
Let $\omega$ be a weight function such that $\omega(2t) = O(\omega(t))$ and $\omega(e^t)$ is convex.
For an entire function $F$ on $\C^d$, if 
	\begin{equation}
		\label{eq:Homogeneous1} 
		\limsup_{s \in \C, ~|s| \to \infty} \frac{\log |F(s u)|}{\omega(|s|)} < \infty ~ \left( \text{resp. } = 0 \right) , \qquad \text{for each } u \in \mathbb{S}^{d-1} , 
	\end{equation}
then,
	\begin{equation}
		\label{eq:Homogeneous2} 
		\limsup_{z \in \C^d, ~|z| \to \infty} \frac{\log |F(z)|}{\omega(|z|)} < \infty ~ \left( \text{resp. } = 0 \right) . 
	\end{equation}
\end{proposition}

The strength of this result is demonstrated by the fact that, in general, coordinate-wise time-frequency decay does not suffice for \eqref{eq:NearlyOptimalTFDecaySpecifiedHermite} to hold, as in the latter case, Proposition \ref{t:Homogeneous} cannot be applied (and we can only apply the machinery from \cite{N-T-V-HermExpSpFuncNearlyOptimalTFDecay} in such a case).

We end this Introduction by detailing the structure of this paper.
In Section \ref{sec:WeightFunc}, we formally introduce weight functions and define several indices associated with them.
These indices correspond to specific properties related to the weight function and will enable us to make precise estimations solely in terms of the weight function.
In particular, we consider the applicability of the Phragm\'en-Lindel\"of principle on sectors relative to the weight function.
Then, in Section \ref{sec:PreciseEst}, we show the correspondence between (uniform) time-frequency decay and skewed Hermite coefficient bounds. 
We prove there Proposition \ref{t:Homogeneous}, as well as Theorem \ref{t:CharWeightedTFDecay} and part of Theorem \ref{t:CharNearlyOptimalTFDecay}; furthermore, we will actually establish finer quantified versions of all these results where more precise estimates are exhibited by using the various indices associated with the weight functions.
The final Section \ref{sec:Coordinates} is devoted to studying coordinate-wise time-frequency decay.
We provide general estimates, discuss when uniform time-frequency decay may be exchanged with coordinate-wise analogs, finishing the proof of Theorem \ref{t:CharNearlyOptimalTFDecay} in the process, and establish when they are not interchangeble.
	
\section{Weight functions and their associated indices}
\label{sec:WeightFunc}

In this text, the term weight function, often denoted by $\omega$, stands for an unbounded non-decreasing function $[0, \infty) \to [0, \infty)$.
Furthermore, we will impose one or more of the following conditions on $\omega$ at $\infty$ ($\sigma > 0$):
	\begin{itemize}
		\item[$(\alpha)$] $\omega(t + s) \leq L[\omega(t) + \omega(s) + 1]$ for some $L \geq 1$ and for all $t, s \geq 0$;
		\item[$(\beta_\sigma)$] $\int_1^\infty \omega(s) s^{-1-\sigma}ds < \infty$;
		\item[$(\beta_\sigma^*)$] $\int_1^\infty \omega(ts) s^{-1-\sigma}ds = O(\omega(t))$; 
		\item[$(\delta)$] the function $\varphi(u) = \omega(e^u)$ is convex.
	\end{itemize}
	
Two weight functions $\omega$ and $\chi$ are called \emph{equivalent} if $\omega(t) = O(\chi(t))$ and $\chi(t) = O(\omega(t))$; in this case we write $\omega \asymp \chi$.

\begin{lemma}
\label{l:WeightFuncBasicProp}
Let $\omega$ be a weight function and $\sigma > 0$.
	\begin{itemize}
		\item[$(i)$] There exists a smooth weight function $\chi$ such that $\chi \asymp \omega$ and $\chi$ satisfies $(\alpha)$, $(\beta_\sigma)$, $(\beta_\sigma^*)$, and/or $(\delta)$ whenever $\omega$ does so.
		\item[$(ii)$] If $\omega$ satisfies $(\beta_\sigma^*)$, then $\omega$ satisfies $(\alpha)$ and $(\beta_\sigma)$.
		\item[$(iii)$] If $\omega$ satisfies $(\beta_\sigma)$, then $\omega(t) = o(t^\sigma)$.
	\end{itemize}
\end{lemma}

\begin{proof}
See Lemmas 2.1, 2.2, and 2.3 in \cite{N-T-V-HermExpSpFuncNearlyOptimalTFDecay}.
\end{proof}

Suppose $\omega$ is a weight function satisfying $(\delta)$.
We consider the \emph{Young conjugate} $\varphi^*$ of the function $\varphi$, defined as
	\[ \varphi^*(v) = \sup_{u \geq 0} [u v - \varphi(u)] , \qquad v \geq 0 . \]
Then $\varphi^*$ is convex, unbounded, non-decreasing, $(\varphi^*)^* = \varphi$, and the function $v \mapsto \varphi^*(v) / v$ is non-decreasing.
Because of $(\delta)$ we note that either 
	\[ \omega \asymp \log_+ = \max\{0, \log\} \quad \text{or} \quad \log t = o(\omega(t)) . \]
In the first case, we have $\varphi^*(t) = \infty$ identically on some interval ($t_0, \infty)$ and is finite valued on $[0, t_0]$.
In the second case, $\varphi^*$ is finite-valued and $t = o(\varphi^*(t))$.

We now explicitly consider the Young conjugates for two types of weight functions, which will appear as examples throughout this article.

\begin{example}
Let $a > 0$.

$(i)$ For $\omega(t) = t^a$ and $\varphi_{t^a}(u) = e^{au}$, we find the Young conjugate
	\begin{equation} 
		\label{eq:YoungConjugate-t^a}
		\varphi^*_{t^a}(v) = \frac{1}{a} v \log\left(\frac{v}{ae}\right) . 
	\end{equation}
By Stirling's formula
	\begin{equation}
		\label{eq:Stirling}
		n! = \sqrt{2 \pi n} \left(\frac{n}{e}\right)^{n} (1 + o(1)) ,
	\end{equation}
we have, for any $0 < s < r$,
	\begin{equation}
		\label{eq:CompSqrtFactAndYoung}
		e^{\frac{1}{r} \varphi^*_{t^a}(r n)} \lesssim \left(\frac{r}{a}\right)^{\frac{n}{a}} (n!)^{\frac{1}{a}} 
		\qquad \text{and} \qquad
		(n!)^{\frac{1}{a}} \lesssim \left(\frac{a}{s}\right)^{\frac{n}{a}} e^{\frac{1}{r} \varphi^*_{t^a}(rn)} .
	\end{equation}

$(ii)$ For $\omega(t) = (\log_+ t)^{1+a}$ and $\varphi_{(\log_+ t)^{1+a}}(u) = u^{1+a}$, we find the Young conjugate
	\begin{equation}
		\label{eq:YoungConjugate-log^1+a}
		\varphi^*_{(\log_+ t)^{1+a}}(v) = \left(\frac{1}{a}\right)^{\frac{1}{a}} \left( \frac{a}{1+a} v \right)^{\frac{1+a}{a}}
	\end{equation}
We have, for any $r > 0$,
	\begin{equation}
		\label{eq:YoungSeq-log^1+a}
		e^{\frac{1}{r} \varphi^*_{(\log_+ t)^{1+a}}(r n)} = e^{\left(\frac{r}{a}\right)^{\frac{1}{a}} \left( \frac{a}{1+a} n \right)^{\frac{1+a}{a}}} .
	\end{equation}
\end{example}

In the remainder of this section, we introduce several indices for $\omega$ related to the conditions $(\alpha)$ and $(\beta^*_\sigma)$.
These indices will allow us to make precise estimates in connection to our quantitative uncertainty principles.
We note, however, that these indices are not invariant under equivalences.

\subsection{Moderateness}
We consider the moderate behavior of a weight function $\omega$ when it satisfies $(\alpha)$.
We introduce the following indices associated with a weight function.

\begin{definition}
Let $\omega$ be a weight function. 
We write
	\[ \ModOm = \inf \{ L \in [1, \infty] \mid \exists C > 0 ~ \forall t, s \geq 0 \, : \, \omega(t + s) \leq L(\omega(t) + \omega(s)) + C \} \]
and, for any $\kappa > 0$,
	\[ \ModKapOm = \inf \{ L \in [1, \infty] \mid \exists C > 0 ~ \forall t \geq 0 \, : \, \omega(\kappa t) \leq L \omega(t) + C \} . \]
Note that $\ModKapOm$ is non-decreasing in $\kappa$ and $\ModKOm{1} = 1$.
\end{definition}

The proofs of the following two results are straightforward, and we therefore omit them.

\begin{lemma}
\label{l:AlphaEquiv}
For a weight function $\omega$, the following are equivalent:
	\begin{itemize}
		\item[$(i)$] $\omega$ satisfies $(\alpha)$;
		\item[$(ii)$] $\ModOm < \infty$;
		\item[$(iii)$] $\ModKapOm < \infty$ for some/all $\kappa > 1$.
	\end{itemize}
If $\omega$ satisfies $(\alpha)$, we have $\ModOm \leq \ModKOm{2} \leq 2 \ModOm$ and $\kappa \mapsto \ModKapOm$ is a non-decreasing submultiplicative function, i.e., $\ModKOm{\kappa_1 \kappa_2} \leq \ModKOm{\kappa_1} \ModKOm{\kappa_2}$ for all $\kappa_1, \kappa_2 > 0$.
\end{lemma}

\begin{corollary}
\label{c:alpha=1}
For any weight function $\omega$ satisfying $(\alpha)$, if $\ModKOm{\kappa_0} = 1$ for some $\kappa_0 > 1$, then, also $\ModOm = \ModKapOm = 1$ for all $\kappa > 0$.
\end{corollary}

More information is available when the weight function is either concave or convex.

\begin{lemma}
\label{l:ConcaveConvex}
Let $\omega$ be a weight function satisfying $(\alpha)$.
	\begin{itemize}
		\item[$(i)$] If $\omega$ is concave, then $\ModOm = 1$, $\ModKapOm \geq \kappa$ for $0 < \kappa \leq 1$, and $\ModKapOm \leq \kappa$ for $\kappa \geq 1$.
		\item[$(ii)$] If $\omega$ is convex, then $\ModOm = \frac{1}{2} \ModKOm{2}$, $\ModKapOm \leq \kappa$ for $0 < \kappa \leq 1$, and $\ModKapOm \geq \kappa$ for $\kappa \geq 1$.
	\end{itemize}
\end{lemma}

When the weight function $\omega$ additionally satisfies $(\delta)$, we have the following.

\begin{lemma}
\label{l:alpha->0}
Let $\omega$ be a weight function satisfying $(\alpha)$ and $(\delta)$.
Then 
	\begin{equation}
		\label{eq:alpha->0}  
		\lim_{\kappa \to 1} \ModKapOm = 1 . 
	\end{equation}
In particular, the function $\kappa \mapsto \ModKapOm$ is continuous at all $\kappa > 0$.
\end{lemma}

\begin{proof}
The continuity at each point would be a direct consequence of \eqref{eq:alpha->0} and the fact that the function $\kappa \mapsto \ModKapOm$ is submultiplicative.
It is therefore enough to establish \eqref{eq:alpha->0}.
Let $0 < \nu_0 < \nu_1$.
Take any $L > \ModKOm{e^{\nu_1}}$ and let $C > 0$ be such that $\omega(e^{\nu_1} t) \leq L \omega(t) + C$ for $t \geq 0$.
As $\varphi$ is non-decreasing and convex, we find, for $s \geq 0$,
	\[ \frac{\varphi(s + \nu_0) - \varphi(s)}{\nu_0} \leq \frac{\varphi(s + \nu_1) - \varphi(s + \nu_0)}{\nu_1 - \nu_0} . \]
Then, for $s \geq 0$,
	\[ \varphi(s + \nu_0) \leq \frac{\nu_1 - \nu_0}{\nu_1} \varphi(s) +  \frac{\nu_0}{\nu_1} \varphi(s + \nu_1) \leq \left( 1 + \frac{\nu_0}{\nu_1} L \right) \varphi(s) + \frac{\nu_0}{\nu_1} C . \]
This shows
	\[ \ModKOm{e^{\nu_0}} \leq 1 + \frac{\ModKOm{e^{\nu_1}}}{\nu_1} \nu_0 . \]
Fixing $\nu_1$ and letting $\nu_0 \to 0^{+}$ already shows $\lim_{\kappa \to 1^+} \ModKapOm = 1$.
On the other hand, since $(\ModKOm{1/\kappa})^{-1} \leq \ModKapOm \leq 1$ for any $0 < \kappa < 1$, the result follows.
\end{proof}

\begin{lemma}
\label{l:AddingSmallRegularity}
Let $\omega$ be a weight function satisfying $(\alpha)$ and $(\delta)$.
For any $\kappa > 0$ and $L > \ModKapOm$ there is a $C_L \geq 1$ such that for every $r > 0$:
	\[ \kappa^n e^{\frac{L}{r} \varphi^*\left(\frac{r}{L} n\right)} \leq C_L^{\frac{1}{r}} e^{\frac{1}{r} \varphi^*(rn)} , \qquad n \in \N . \]
\end{lemma}

\begin{proof}
Take any $\kappa > 0$ and $L > \ModKapOm$ and let $C_L \geq 1$ be such that $\omega(\kappa t) \leq L \omega(t) + \log C_L$ for all $t \geq 0$.
Then, for any $u \geq \log \kappa$,
	\[ \varphi(u) = \varphi(\log \kappa + u - \log \kappa) = \omega(\kappa \cdot e^{u - \log \kappa}) \leq L \omega(e^{u - \log \kappa}) + \log C_L = L \varphi(u - \log \kappa) + \log C_L . \]
Therefore, for $n \in \N$,
	\begin{align*}
		n \log \kappa + \frac{L}{r} \varphi^*\left(\frac{r n}{L}\right)
		&= \sup_{u \geq 0} \left( (\log \kappa + u) n - \frac{L}{r} \varphi(u) \right)
		= \sup_{u \geq \log \kappa} \left( u n - \frac{L}{r} \varphi(u - \log \kappa) \right) \\
		&\leq \sup_{u \geq 0} \left( u n - \frac{1}{r} \varphi(u) \right) + \frac{1}{r} \log C_L
		\leq \frac{1}{r} \varphi^*(r n) + \frac{1}{r} \log C_L .
	\end{align*}
\end{proof}

\begin{corollary}
\label{c:SeqSummable}
Let $\omega$ be a weight function satisfying $(\alpha)$ and $(\delta)$.
Suppose $(c_\alpha)_{\alpha \in \N^d} \in \C^{\N^d}$ is such that $\sup_{\alpha \in \N^d} |c_\alpha| e^{\frac{1}{r} \varphi^*(r |\alpha|)} < \infty$ for some $r > 0$.
Then
	\[ \sum_{\alpha \in \N^d} |c_\alpha| e^{\frac{1}{s} \varphi^*(s |\alpha|)} < \infty , \qquad s < r . \]
\end{corollary}

\begin{proof}
Take any $s < r$. Then $s < r / \ModKapOm$ for some $\kappa > 1$ by \eqref{eq:alpha->0}.
Therefore, $e^{\frac{1}{r} \varphi^*(r |\alpha|) - \frac{1}{s} \varphi^*(s |\alpha|)} \lesssim \kappa^{-|\alpha|}$ for $\alpha \in \N^d$ due to Lemma \ref{l:AddingSmallRegularity}.
Consequently
	\[ \sum_{\alpha \in \N^d} |c_\alpha| e^{-\frac{1}{s} \varphi^*(s |\alpha|)} \lesssim \sum_{\alpha \in \N^d} e^{\frac{1}{r} \varphi^*(r |\alpha|) - \frac{1}{s} \varphi^*(s |\alpha|)} \lesssim \sum_{\alpha \in \N^d} \kappa^{-|\alpha|} \lesssim 1 . \]
\end{proof}

The next result shows that the bounds in \eqref{eq:NearlyOptimalTFDecaySpecifiedHermite} are almost non-increasing.

\begin{corollary}
\label{c:DecayHermiteBounds}
Let $\omega$ be a weight function satisfying $(\alpha)$, $(\delta)$, and $\omega(t) = o(t^2)$.
For any $r > 0$ there exists a $C = C_r > 0$ such that
	\[  \sqrt{k!} e^{-\frac{1}{r} \varphi^*(r k)} \leq C 2^{\frac{m}{2}} \sqrt{m!} e^{-\frac{1}{r} \varphi^*(r m)} , \qquad 0 \leq m \leq k . \]
\end{corollary}

\begin{proof}
Take any $r > 0$.
Since $\omega(t) = o(t^2)$, we have by \eqref{eq:CompSqrtFactAndYoung}
	\[ \sqrt{n!} \lesssim 2^{-\frac{n}{2}} e^{\frac{1}{5} \varphi^*_{t^2}(5 n)} \lesssim 2^{-\frac{n}{2}} e^{\frac{1}{r} \varphi^*(r n)} . \]
Then, by the convexity of $\varphi^*$, for any $0 \leq m \leq k$, 
	\[
		\frac{\sqrt{k!}}{\sqrt{m!}} e^{\frac{1}{r} \varphi^*(rm)} 
		\leq 2^{\frac{k}{2}} \sqrt{(k - m)!} e^{\frac{1}{r} \varphi^*(rm)} 
		\lesssim 2^{\frac{m}{2}} e^{\frac{1}{r} \varphi^*(r (k - m)) + \frac{1}{r} \varphi^*(r m)} 
		\leq 2^{\frac{m}{2}} e^{\frac{1}{r} \varphi^*(r k)} .
	\]
\end{proof}

We end this subsection by considering the $\mathfrak{m}$-indices for specific weight functions.
The verification of the following properties is straightforward.

\begin{lemma}
\label{l:ExamplesAlpha}~\\
\indent
$(i)$ The weight function $t^a$, $a > 0$, satisfies $(\alpha)$ and $(\delta)$ and
	\[ \Mod{t^a} = \max \{2^{a - 1}, 1 \} , \qquad \ModK{\kappa}{t^a} = \kappa^a , \quad \kappa > 0 . \]
	
\indent
$(ii)$ The weight function $(\log_+ t)^{1 + a}$, $a \geq 0$, satisfies $(\alpha)$ and $(\delta)$ and
	\[ \Mod{(\log_+ t)^{1 + a}} = \ModK{\kappa}{(\log_+ t)^{1 + a}} = 1 , \qquad \kappa > 0 . \]
\end{lemma}

\subsection{The Phragm\'{e}n-Lindel\"{o}f principle on sectors}

For any $\theta > 0$, we consider the sector\footnote{If $\theta \geq 2\pi$, it is a subset of the Riemann surface of the logarithm, where its boundary should be taken.}
	\[  S_\theta = \left\{ z \in \C \mid - \frac{\theta}{2} < \arg z < \frac{\theta}{2} \right\} . \]
One may formulate the classical Phragm\'{e}n-Lindel\"{o}f principle on sectors as follows.
Throughout the sequel, let $C(\Omega)$ be the set of continuous functions on $\Omega$. 
If in addition $\Omega$ is open, then $\mathcal{A}(\Omega)$ denotes the set of all analytic functions on $\Omega$.

\begin{lemma}
	\label{l:PLClassic}
	Let $\theta > 0$.
	Suppose $F \in C(\overline{S}_\theta)$ is such that $\log |F(z)|$ is subharmonic on $S_\theta$ and for which, for some $M > 0$,
		\[ |F(z)| \leq M , \qquad z \in \partial S_\theta , \]
	and, for every $\varepsilon > 0$,
		\[ |F(z)| \lesssim_\varepsilon e^{\varepsilon |z|^{\frac{\pi}{\theta}}} , \qquad z \in S_\theta . \]
	Then
		\[ |F(z)| \leq M , \qquad z \in \overline{S}_\theta . \]
\end{lemma}

\begin{proof}
	For any $\varepsilon > 0$, consider $G_\varepsilon(z) = F(z) e^{- \varepsilon z^{\pi / \theta}} \in C(\overline{S}_\theta)$.
	Then $\log |G_\varepsilon(z)|$ is subharmonic in $S_\theta$.
	Moreover, $|G_\varepsilon(z)| \leq M$ for $z \in \partial S_\theta$ and $|G_\varepsilon(x)| = o(1)$ as $x \to \infty$ on $[0, \infty)$.
	We apply the Phragm\'en-Lindel\"of theorem as in \cite[p.~25]{K-LogIntI} (and the remark following it) on the sectors $\{ z \in \C \mid -\frac{\theta}{2} < \arg z < 0 \}$ and $\{ z \in \C \mid 0 < \arg z < \frac{\theta}{2} \}$.
	This shows that 
		\[ |G_\varepsilon(z)| \leq \max \{ M, \sup_{x > 0} |G_\varepsilon(x)| \} , \qquad z \in S_\theta . \]
	However, since $\sup_{x \geq 0}\log |G_\varepsilon(x)|$ is reached at some point on the real line, the maximum modulus principle and $G_\varepsilon(x) = o(1)$ imply that $|G_\varepsilon(z)|$ is bounded by $M$. 
	Therefore,
		\[ |F(z)| = \lim_{\varepsilon \to 0^+} |G_\varepsilon(z)| \leq M , \qquad z \in S_\theta . \]
\end{proof}

One of our primary tools will be the application of weighted forms of the Phragm\'{e}n-Lindel\"{o}f principle for analytic functions on sectors.
To this end, we introduce the following index.

\begin{definition}
Let $\omega$ be a weight function.
\begin{itemize}
\item[$(i)$] $\omega$ is called \emph{Phragm\'en-Lindel\"of admissible of angle $\theta > 0$ and order $L \in [1, \infty]$}
if for any $\varrho \in (0, \theta]$ and $\lambda > 0$ there is a $C_{\theta, L} > 0$ such that if $F \in \mathcal{A}(S_\varrho) \cap C(\overline{S}_\varrho)$ satisfies
	\begin{equation}
		\label{eq:PLBoundCond} 
		|F(z)| \leq M_\lambda e^{\lambda \omega(|z|)} , \qquad z \in \partial S_\varrho , 
	\end{equation}
and, for all $\varepsilon > 0$,
	\begin{equation} 
		\label{eq:PLSectorCond}
		F(z)| \lesssim_{\varepsilon} e^{\varepsilon |z|^{\frac{\pi}{\varrho}}} , \qquad z \in S_\varrho , 
	\end{equation}
then
	\[ |F(z)| \leq C_{\theta, L} M_\lambda e^{L \lambda \omega(|z|)} , \qquad z \in S_\varrho . \]

\item[$(ii)$] The \emph{Phragm\'en-Lindel\"of index of angle $\theta$} of $\omega$ is denoted by	
	\[
		\PL_\theta(\omega) := \inf \{ L \geq 1 \mid \omega \text{ is Phragm\'en-Lindel\"of admissible of angle } \theta \text{ and order } L \} . 
	\]
\end{itemize}
\end{definition}

We note that a weight function $\omega$ is trivially Phragm\'en-Lindel\"of admissible of angle $\theta$ and order $\infty$.
Therefore $\PL_\theta(\omega)$ is well-defined for any angle $\theta$.
Moreover, $\PL_\theta(\omega)$ is clearly non-decreasing in $\theta$.

A key insight in \cite{N-T-V-HermExpSpFuncNearlyOptimalTFDecay} was that the finiteness of $\PL_\theta(\omega)$---and thus applicability of the Phragm\'{e}n-Lindel\"{o}f principle with respect to $\omega$ on sectors of angle at most $\theta$---is equivalent to $\omega$ satisfying $(\beta_{\frac{\pi}{\theta}}^*)$. 
To then obtain general bounds on $\PL_\theta(\omega)$ in terms of more direct estimates for $\omega$, we also introduce the following index.

\begin{definition}
Let $\omega$ be a weight function and $\sigma > 0$.
We write
	\[ \beta^*_\sigma(\omega) = \inf \{ L \in [1, \infty] \mid \exists C > 0 ~ \forall t \geq 0 \, : \, \sigma \int_{1}^\infty \frac{\omega(s t)}{s^{1 + \sigma}} ds \leq L \omega(t) + C \} . \]
\end{definition}

\begin{lemma}
	\label{l:PLBounds}
	Let $\sigma > 0$ and $\omega$ be a weight function satisfying $(\alpha)$ and $(\beta_\sigma)$.
	The following are true:
		\begin{itemize}
			\item[(i)] $\PL_{\frac{\pi}{\sigma}}(\omega) < \infty$ if and only if $\omega$ satisfies $\beta^*_\sigma(\omega) < \infty$;
			\item[(ii)] If $(i)$ holds and in addition $\omega$ is differentiable satisfying $(\delta)$, then
				\begin{equation}
					\label{eq:PLBounds} 
					\frac{1}{\pi} \beta^*_\sigma(\omega) \leq \PL_{\frac{\pi}{\sigma}}(\omega) \leq \left( 1 + \frac{2}{\pi} \beta^*_\sigma(\omega) \right) \ModKOm{\sqrt{2}} . 
				\end{equation}
		\end{itemize}
\end{lemma}

\begin{proof}
Statement (i) was shown in \cite[Theorem 3.1]{N-T-V-HermExpSpFuncNearlyOptimalTFDecay}.
To prove (ii), we follow its proof and make adjustments where necessary.

Set $\omega_\sigma(t) = \omega(t^{1 / \sigma})$.
Consider the harmonic function
	\[ P_{\omega_\sigma}(z) = \frac{x}{\pi} \int_{-\infty}^{\infty} \frac{\omega_\sigma(|t|)}{x^2 + (y - t)^2} dt , \qquad z = x + iy \in S_\pi . \]
As $\omega$ is continuous, it follows that $\lim_{x \to 0^+} P_{\omega_\sigma}(z) = \omega_\sigma(|y|)$. 
Therefore, we may continuously extend $P_{\omega_\sigma}$ to $\overline{S}_\pi$ and we have $P_{\omega_\sigma}(iy) = \omega_\sigma(|y|)$.
Additionally, if $V_{\omega_\sigma}$ is a harmonic conjugate of $P_{\omega_\sigma}$ on $S_\pi$, then one shows (cfr. the proof of \cite[Theorem 3.1]{N-T-V-HermExpSpFuncNearlyOptimalTFDecay}) that $V_{\omega_\sigma}$ may be continuously extended to a function on $\overline{S}_\pi$.

We now consider several estimates involving $P_{\omega_\sigma}$.
By \cite[Lemma 2.2]{B-M-T-UltradiffFuncFourierAnal}, there exists, for every $\varepsilon > 0$, some $C_\varepsilon > 0$ such that
	\[ |P_{\omega_\sigma}(z)| \leq \varepsilon |z| + C_\varepsilon , \qquad z \in S_\pi . \]
Since $\omega$ satisfies $(\delta)$, the function $z \mapsto \omega_\sigma(|z|) = \varphi\left(\frac{1}{\sigma} \log |z|\right)$ is subharmonic on $\overline{S}_\pi$, being the composition of a convex non-decreasing function and a subharmonic function.
Therefore, $\omega_\sigma - P_{\omega_\sigma}$ is a subharmonic function on $\overline{S}_\pi$ that vanishes on the imaginary line.
By Lemma \ref{l:PLClassic}, we have 
	\[ \omega_\sigma(|z|) \leq P_{\omega_\sigma}(z) , \qquad z \in S_\pi . \]
Note that, for $x > 0$,
	\begin{align*}
		\frac{x}{\pi} \int_{-\infty}^\infty \frac{\omega_\sigma(|t|)}{x^2 + t^2} dt 
		&\geq \frac{2}{\pi} \int_1^\infty \frac{\omega_\sigma(x t)}{1 + t^2} dt
		= \frac{2 \sigma}{\pi} \int_1^\infty \frac{\omega(x^{1 / \sigma} s)}{s^{1 - \sigma} + s^{1 + \sigma}} ds
		\geq \frac{\sigma}{\pi} \int_1^\infty \frac{\omega(x^{1 / \sigma} s)}{s^{1 + \sigma}} ds ,
	\end{align*}
that is
	\[ \frac{\sigma}{\pi} \int_1^\infty \frac{\omega(x s)}{s^{1 + \sigma}} ds \leq P_{\omega_\sigma}(x^\sigma) , \qquad x > 0 . \]
Suppose now $L _0, L_1 \geq 1$ and $C > 0$ are such that
	\[ \sigma \int_1^\infty \frac{\omega(t s)}{s^{1 + \sigma}} ds \leq L_0 \omega(t) + C , \qquad \omega(\sqrt{2} t) \leq L_1 \omega(t) + C , \qquad t \geq 0 . \]
Then, for $z = x + iy$ with $x, y \geq 0$ and $x + y \geq 1$ we have
	\begin{align*}
		P_{\omega_\sigma}(z)
		&= \frac{x}{\pi} \int_{|t| \leq x + y} \frac{\omega_\sigma(|t|)}{x^2 + (y - t)^2} dt + \frac{x}{\pi} \int_{|t| > x + y} \frac{\omega_\sigma(|t|)}{x^2 + (y - t)^2} dt \\
		&\leq \omega_\sigma(x + y) + \frac{2}{\pi} \int_1^\infty \frac{\omega_\sigma(xs + y)}{1 + s^2} ds
		\leq \omega_\sigma(x + y) + \frac{2}{\pi} \int_1^\infty \frac{\omega_\sigma([x + y]s)}{1 + s^2} ds \\
		&\leq \omega_\sigma(x + y) + \frac{2 \sigma}{\pi} \int_1^\infty \frac{\omega([x + y]^{1 / \sigma} s)}{s^{1 + \sigma}} ds
		\leq \left( 1 + \frac{2 L_0}{\pi} \right) \omega_\sigma(x + y) + \frac{2 C \sigma}{\pi} \\
		&\leq \left( 1 + \frac{2 L_0}{\pi} \right) L_1 \omega_\sigma(|z|) + \left( 1 + \frac{2 L_0}{\pi} + \frac{2 \sigma}{\pi} \right) C .
	\end{align*}
One shows the same bound when $x \geq 0$, $y < 0$, and $x - y \geq 1$.
The remaining points form a compact subset of $\overline{S}_\pi$, so by the continuity of $P_{\omega_\sigma}(z)$, we conclude that, for some $C_{\theta, L} > 1$,
	\[ P_{\omega_\sigma}(z) \leq \left( 1 + \frac{2 L_0}{\pi} \right) L_1 \omega_\sigma(|z|) + \frac{1}{\lambda} \log C_{\theta, L} , \qquad z \in S_\pi . \]
	
We are now in the position to finish the proof. Consider the function
	\[ F_{\omega}(z) = \exp[ P_{\omega_\sigma}(z^\sigma) + i V_{\omega_\sigma}(z^\sigma)] , \qquad z \in \overline{S}_{\pi / \sigma} , \]
which is analytic on $S_{\pi / \sigma}$ with continuous extension to its boundary.
On the one hand, suppose $F \in \mathcal{A}(S_{\pi / \sigma}) \cap C(\overline{S}_{\pi / \sigma})$ is such that \eqref{eq:PLBoundCond} and \eqref{eq:PLSectorCond} hold, for certain $M, \lambda > 0$.
Then $G(z) = F(z) F_\omega^{-\lambda}(z)$ is analytic function on $S_{\pi / \sigma}$ with continuous extension bounded by $M$ on $\partial S_{\pi / \sigma}$ and such that \eqref{eq:PLSectorCond} is true.
Applying Lemma \ref{l:PLClassic} gives us 
	\[ |F(z)| \leq M |F_{\omega}^\lambda(z)| \leq C_{\theta, L} M \exp\left[ \left( 1 + \frac{2 L_0}{\pi} \right) L_1 \lambda \omega(|z|) \right] , \qquad z \in S_{\frac{\pi}{\sigma}} . \]
This shows the upper bound for $\PL_{\frac{\pi}{\sigma}}(\omega)$.
On the other hand, if $L > \PL_{\frac{\pi}{\sigma}}(\omega)$, then applying the Phragm\'{e}n-Lindel\"{o}f principle to $F_{\omega}$ with respect to $\omega$ on the sector $S_{\frac{\pi}{\sigma}}$ gives us, for some $C_F > 0$,
	\[ \frac{\sigma}{\pi} \int_1^\infty \frac{\omega(x s)}{s^{1 + \sigma}} ds \leq P_{\omega_\sigma}(x^\sigma) = \log |F_{\omega}(x)| \leq L \omega(x) + C_F , \qquad x > 0 . \]
This shows the lower bound for $\PL_{\frac{\pi}{\sigma}}(\omega)$.
\end{proof}

For concrete examples of weight functions $\omega$, however, we may often explicitly calculate $\PL_\theta(\omega)$ by determining exactly the analytic function appearing in the proof of Lemma \ref{l:PLBounds}.

\begin{lemma}
\label{l:ExamplesPL}~\\
\indent
$(i)$ $\PL_\theta(t^a) = \sec\left(\frac{a \theta}{2}\right)$ for $a > 0$ and $\theta \in (0, \frac{\pi}{a})$.

$(ii)$ $\PL_\theta((\log_+ t)^{1 + a}) = 1$ for $a \geq 0$ and $\theta > 0$.
\end{lemma}

\begin{proof}
$(i)$ For $0 < \theta < \frac{\pi}{a}$, we define the analytic function
	\[ F_{a, \theta}(z) = \exp\left[ \sec\left(\frac{a \theta}{2}\right) z^a \right] . \]
Then
	\[ |F_{a, \theta}(r e^{\pm \frac{i \theta}{2}})| = e^{r^a} , \qquad r \geq 0 . \]
Take now any analytic function $F$ on the sector $S_{\theta}$ with continuous extension to $\overline{S}_{\theta}$ satisfying \eqref{eq:PLBoundCond} with respect to $\omega(t) = t^a$, for some $\lambda, M > 0$, and \eqref{eq:PLSectorCond}.
Consider the analytic function $G = F F^{-\lambda}_{a, \theta}$ on the sector $S_{\theta}$ with analytic extension to $\overline{S}_{\theta}$.
Moreover, we have that $|G(r e^{\pm \frac{i \theta}{2}})| \leq M$ for $r \geq 0$.
Applying Lemma \ref{l:PLClassic} to $G$ on the sector $S_{\theta}$, we find that $|G(z)| \leq M$ for $z \in \overline{S}_{\theta}$.
Consequently,
	\[ |F(z)| = |G(z) F^{\lambda}_{a, \theta}(z)| \leq M e^{\sec\left(\frac{a \theta}{2}\right) \lambda |z|^a} , \qquad z \in \overline{S}_{\theta} . \]
This already shows that $\PL_\sigma(t^a) \leq \sec\left(\frac{a \theta}{2}\right)$ for $\theta \in (0, \pi/a)$.
However, since
	\[ |F_{a, \theta}(r)| = \exp\left[ \sec\left(\frac{a \theta}{2}\right) r^a \right] , \qquad r > 0 , \]
we, in fact, have equality.

$(ii)$ Take any $\theta > 0$. Consider the analytic function
	\[ F_a(z) = \exp[ \{\log(1 + z)\}^{1 + a} ] , \qquad z \in S_{\theta} , \]
which has continuous extension to its boundary. 
Since the argument of any element of $S_\theta$ is bounded, one sees that, for every $\varepsilon > 0$,
	\[ |F_a(z)| \geq \exp[ (1 - \varepsilon) (\log |z|)^{1 + a} ] \]
for $z$ large enough in $S_{\theta}$.
Suppose $F$ is an analytic function on the sector $S_{\theta}$ with continuous extension to $\overline{S}_{\theta}$ satisfying \eqref{eq:PLBoundCond} with respect to $\omega(t) = (\log_+ t)^{1+a}$, for some $\lambda, M > 0$, and \eqref{eq:PLSectorCond}.
Then, for any $\lambda' > \lambda$, we have $G(z) = F F^{-\lambda'}_a$ is bounded by $M B_{\lambda'}$ on $\partial S_\theta$, for some $C_{\lambda'}$ only depending on $\lambda'$.
As in $(i)$, we have that
	\[ |F(z)| \leq C_{\lambda'} M |F^{\lambda'}_a(z)| , \qquad z \in \overline{S}_{\theta} . \]
From here, for every $\varepsilon > 0$, one finds a $C_\varepsilon > 0$ (independent of $F$ and $\lambda'$) such that
	\[ |F(z)| \leq C_\varepsilon C_{\lambda'} M \exp[ (1 + \varepsilon) \lambda' \log(1 + |z|)^{1 + a} ] , \qquad z \in \overline{S}_\theta . \]
This shows the result as $\lambda' > \lambda$ and $\varepsilon > 0$ were chosen arbitrarily.
\end{proof}

\section{Quantitative uncertainty principles for uniform time-frequency decay}
\label{sec:PreciseEst}

In this section, we consider how uniform time-frequency decay implies bounds on the skewed Hermite coefficients---and vice-versa.
In particular, we establish several of the results stated in the Introduction, albeit with greater precision.
Throughout this section, $A \in \GL_+(\R^d)$.
Our first result looks at Gaussian time-frequency decay with scaling factor close to the optimal value $\frac12$.

\begin{proposition}
\label{p:t^2Est}
If $f \in L^2(\R^d)$ satisfies
	\begin{equation}
		\label{eq:TFDecayt^2}
		|f(x)| \lesssim e^{- (\frac12 - \lambda) \ev{Ax}{x}} \quad \text{and} \quad |\widehat{f}(\xi)| \lesssim e^{- (\frac12 - \lambda) \ev{A^{-1} \xi}{\xi}} ,
	\end{equation}
for some $0 < \lambda < 1/2$, then,
	\begin{equation}
		\label{eq:t^2EstHermite}
		|H_A(f, \alpha)| \lesssim_h h^{|\alpha|} , \qquad h > d \left( 2 \lambda \right)^{\frac14} .
	\end{equation}
\end{proposition}

When \eqref{eq:NearlyOptimalTFDecay} holds, we may further specify using weight functions.

\begin{proposition}
\label{p:PreciseEstGenCase}
Let $\omega$ be a weight function satisfying $(\beta^*_2)$ and $(\delta)$.
If $f \in L^2(\R^d)$ satisfies
	\[
		|f(x)| \lesssim e^{-\frac{1}{2} \ev{Ax}{x} + \lambda \omega(|A^{1/2} x|)} \quad \text{and} \quad |\widehat{f}(\xi)| \lesssim e^{-\frac{1}{2} \ev{A^{-1} \xi}{\xi} + \lambda \omega(|A^{-1/2} \xi|)} , 
	\]
for some $\lambda > 0$, then,
	\[ 
		|H_A(f, \alpha)| \lesssim_r \sqrt{\alpha!} e^{-\frac{1}{r} \varphi^*(r |\alpha|)} , \qquad 0 < r < (\mathfrak{C}_{1, d}(\omega) \lambda)^{-1} ,
	\]
with 
	\[ \mathfrak{C}_{1, d}(\omega) := \ModOm \ModKOm{\frac{d}{\sqrt{2}}} \PL_{\frac{\pi}{2}}(\omega) . \]
\end{proposition}

\begin{remark}
\label{rem:ExamplesH}
Using Lemmas \ref{l:ExamplesAlpha} and \ref{l:ExamplesPL}, we may explicitely calculate $\mathfrak{C}_{1, d}(\omega)$ for several important well-known weight functions:

$(i)$ $\mathfrak{C}_{1, d}(t^a) = d^a 2^{\frac{a}{2} - 1} \sec\left(\frac{a \pi}{4}\right)$ for $1 \leq a < 2$;

$(ii)$ $\mathfrak{C}_{1, d}(t^a) = d^a 2^{-\frac{a}{2}} \sec\left(\frac{a \pi}{4}\right)$ for $0 < a \leq 1$;

$(iii)$ $\mathfrak{C}_{1, d}((\log_+ t)^{1 + a}) = 1$ for $a \geq 0$.
\end{remark}

Next, we consider the reverse problem, namely, how bounds on the skewed Hermite coefficients imply time-frequency decay bounds.
We point out that the ensuing results are essentially shown in \cite[Theorems 5.1 and 5.2]{N-T-V-HermExpSpFuncNearlyOptimalTFDecay}.
However, as before, we make them more precise by giving quantitative versions of such results.  

\begin{proposition}
\label{p:HermiteToTFEstFort^2}
If $f \in L^2(\R^d)$ satisfies
	\[ 
		|H_A(f, \alpha)| \lesssim h^{|\alpha|} , 
	\]
for some $h < \frac{1}{2\sqrt{d}}$, then, 
	\[
		|f(x)| \lesssim_\lambda e^{- (\frac12 - \lambda) \ev{Ax}{x}} \quad \text{and} \quad |\widehat{f}(\xi)| \lesssim_\lambda e^{- (\frac12 - \lambda) \ev{A^{-1} \xi}{\xi}} , \qquad \lambda > 2 d h^2 . 
	\]
\end{proposition}

For functions satisfying \eqref{eq:NearlyOptimalTFDecayHermite}, we may also further specify the time-frequency decay using weight functions.

\begin{proposition}
\label{p:HermiteToTFEstGeneral}
Let $\omega$ be a weight function satisfying $(\alpha)$, $(\delta)$, and such that $\omega(t) = o(t^2)$.
If $f \in L^2(\R^d)$ satisfies
	\[
		|H_A(f, \alpha)| \lesssim \sqrt{\alpha!} e^{-\frac{1}{r} \varphi^*(r |\alpha|)} ,	
	\]
for some $r > 0$, then,
	\[
		|f(x)| \lesssim_\lambda e^{-\frac{1}{2} \ev{Ax}{x} + \lambda \omega(|A^{1/2} x|)}  \quad \text{and} \quad |\widehat{f}(\xi)| \lesssim_\lambda e^{-\frac{1}{2} \ev{A^{-1} \xi}{\xi} + \lambda \omega(|A^{-1/2} \xi|)} ,  
	\]
for $\lambda > [\ModOm \ModKOm{\sqrt{2}}] / r$.
\end{proposition}

\begin{remark}
\label{rem:ExamplesH}
Using Lemma \ref{l:ExamplesAlpha}, we may explicitely calculate the coeffcient appearing in Proposition \ref{p:HermiteToTFEstGeneral} for several well-known weight functions:

$(i)$ $\Mod{t^a} \ModK{\sqrt{2}}{t^a} = 2^{\frac32 a - 1}$ for $1 < a < 2$;

$(ii)$ $\Mod{t^a} \ModK{\sqrt{2}}{t^a} = 2^{\frac{a}{2}}$ for $0 < a \leq 1$;

$(iii)$ $\Mod{(\log_+ t)^{1 + a}} \ModK{\sqrt{2}}{(\log_+ t)^{1+a}} = 1$ for $a \geq 0$.
\end{remark}

Theorem \ref{t:CharWeightedTFDecay} is now a direct consequence of Propositions \ref{p:PreciseEstGenCase} and \ref{p:HermiteToTFEstGeneral}.

\begin{proof}[Proof of Theorem \ref{t:CharWeightedTFDecay}]
Since $|A^{1/2} x| \leq \|A\|^{\frac12} |x|$ and $|x| \leq \|A^{-\frac12}\| |A^{1/2} x|$ where $\|\cdot\|$ denotes the matrix operator norm, and similar for $A^{-1}$, the result directly follows from Propositions \ref{p:PreciseEstGenCase} and \ref{p:HermiteToTFEstGeneral} and the property $(\alpha)$ for $\omega$.
\end{proof}

Although we have up to now always employed the $\sup$-norm, the ensuing remark states that it may be interchanged with any other $L^p$-norm.

\begin{remark}
\label{rem:Lpnorms}
Suppose $\log t = o(\omega(t))$ (if not, our weight function is equivalent to $\log_+$ and the remainder of this remark is trivial).
In particular, it follows that $\int_{0}^\infty e^{-\varepsilon \omega(t)} dt < \infty$ for any $\varepsilon > 0$.
By H\"older's inequality, this shows that in each of the above propositions, we may interchange the $\sup$-norm with any $L^p$-norm for $p \in [1, \infty)$.
\end{remark}

Before moving on, we consider more general time-frequency decay of the form, for $A, B \in \GL_+(\R^d)$, 
	\begin{equation}
		\label{eq:TFDecayt^2WithB}
		|f(x)| \lesssim e^{- (\frac12 - \lambda) \ev{Ax}{x}}  \quad \text{and} \quad |\widehat{f}(\xi)| \lesssim e^{- (\frac12 - \lambda) \ev{B \xi}{\xi}} .
	\end{equation}
In fact, we demonstrate how quantitative results as above are only possible in the critical case $B = A^{-1}$.
In what manner it fails depends on whether or not $B^{-1} - A$ is semi-positive definite.
Note that in the case where $A = a I$ and $B = b I$ for $a, b > 0$, we have that $B^{-1} - A$ is semi-positive definite if and only if $ab \leq 1$.
We find the following result.
Its proof is given in Section \ref{sec:ProofBNotInverseA}.

\begin{theorem}
\label{t:AandBNotInverses}
Let $A, B \in \GL_+(\R^d)$ with $B \neq A^{-1}$.

\begin{itemize}
\item[$(i)$] If $B^{-1} - A$ is not semi-positive definite, there exists a $\lambda_0 > 0$ 
such that any $f \in L^2(\R^d)$ satisfying \eqref{eq:TFDecayt^2WithB} for some $\lambda < \lambda_0$ is necessarily null.
On the other hand, there exist non-trivial $f \in L^2(\R^d)$ satisfying \eqref{eq:TFDecayt^2WithB} for $\lambda = \lambda_0$.

\item[$(ii)$] If $B^{-1} - A$ is semi-positive definite, the subspace $\mathcal{V}$ of functions $f \in L^2(\R^d)$ satisfying
	\begin{equation}
		\label{eq:OptimalWhenEigenvalue<1} 
		|f(x)| \lesssim e^{-\frac12 \ev{Ax}{x}} \quad \text{ and } \quad |\widehat{f}(\xi)| \lesssim e^{- \frac12 \ev{B\xi}{\xi}}  , 
	\end{equation}
but for which, for some $h > 0$,
	\begin{equation}
		\label{eq:NoQuantify}
		\sup_{\alpha \in \N^d} \frac{|H_A(f, \alpha)|}{h^{|\alpha|}} = \infty ,
	\end{equation}
is non-trivial.
The subspace $\mathcal{V}$ is dense in $L^2(\R^d)$ if and only if $B^{-1} - A$ is positive definite.
\end{itemize}
\end{theorem}

\begin{remark}
We will see that the threshold $\lambda_0$ appearing in Theorem \ref{t:AandBNotInverses}$(i)$ is exactly $\frac12[1 - \rho(A, B)^{-1/2}]$, where $\rho(A, B)$ denotes the maximal eigenvalue of $B^{1/2} A B^{1/2}$.
\end{remark}

\begin{remark}\label{remarknoquantitativeuncertainty}
	In Theorem \ref{t:AandBNotInverses}, we may again distinguish between two cases.
	
	If $A = \gamma B^{-1}$ for some $\gamma > 0$, then, by defining $A' = (\gamma)^{-1/2} A$ and $B' = (\gamma)^{-1/2} B$, we get $B' = (A')^{-1}$ and we may rescale $f$ and apply the previous propositions for $A'$ and $B'$. These estimates then translate to similar results for the original $A$ and $B$ at the threshold $\lambda = \lambda_0$ for \eqref{eq:TFDecayt^2WithB}.
	
	Otherwise, we may apply an analogous technique as in the proof of $(ii)$ of Theorem \ref{t:AandBNotInverses} to construct $f \in L^2(\R^d)$ satisfying \eqref{eq:TFDecayt^2WithB} for $\lambda = \lambda_0$ and \eqref{eq:NoQuantify} for some $h > 0$.
\end{remark}

The remainder of this section will be dedicated to proofs of all of the above results.

\subsection{The skewed Bargmann transform}
Our main tool for connecting time-frequency decay with specified bounds on the skewed Hermite coefficients will be the (skewed) Bargmann transform \cite{B-HilbertSpAnalFuncAssocIntTrans}.
Formally, on $L^2(\R^d)$, the \emph{$A$-skewed Bargmann transform} is defined as
	\[ \B_{A} : L^2(\R^d) \to \mathcal{F}^2(\C^d) , \qquad f \mapsto \left( z \mapsto \B_{A} f(z) := \sum_{\alpha \in \N^d} \frac{H_A(f, \alpha)}{\sqrt{\alpha!}} z^\alpha  \right) . \]
Here, $\mathcal{F}^2(\C^d)$ denotes the Fock space, that is, the Hilbert space of all entire functions $F$ on $\C^d$ such that $\int_{\C^d} |F(z)|^2 e^{-|z|^2}dz < \infty$.
In particular, for the $A$-skewed Hermite functions, we have
	\[ \B_A h_{A, \alpha}(z) = \frac{z^\alpha}{\sqrt{\alpha!}} , \qquad \alpha \in \N^d . \]
For any $f \in L^2(\R^d)$,
	\begin{equation}
		\label{eq:HermiteCoeffViaBargmann}
		H_A(f, \alpha) = \frac{(\B_A f)^{(\alpha)}(0)}{\sqrt{\alpha!}} , \qquad \alpha \in \N^d . 
	\end{equation}
Explicitely, for $f \in L^2(\R^d)$, we have that $\B_A f$ is given by the integral transform
	\[ \B_A f(z) = \pi^{-\frac{d}{4}} (\det A)^{1/4} \int_{\R^d} f(t) e^{- \frac{1}{2} \left( \ev{z}{z} + \ev{At}{t} \right) + \sqrt{2} \ev{z}{A^{1/2} t}} dt . \]
Indeed, for $A = I$ this was shown in \cite{B-HilbertSpAnalFuncAssocIntTrans}. 
The general case then follows either by the same arguments as in \cite{B-HilbertSpAnalFuncAssocIntTrans}, or simply from a change of variables.

The connection between time-frequency decay and skewed Hermite coefficient bounds follows due to the direct relationship between the skewed Bargmann transform and the short-time Fourier transform \cite{GrochenigBook}.
For a function $f \in L^2(\R^d)$, the \emph{short-time Fourier transform} with window 
	\[ \phi_A(x) = h_{A, 0}(x) = \pi^{-d/4} (\det A)^{1/4} e^{-\frac12 \ev{Ax}{x}} \] 
is given by
	\[ V_{\phi_A} f(x, \xi) = (2 \pi)^{-\frac{d}{{2}}} \int_{\R^d} f(t) \phi_A(t - x) e^{- i \ev{t}{\xi}} dt . \]
Observe that, by the Plancherel identity,
	\begin{equation}
		\label{eq:STFTFourier} 
		V_{\phi_A} f(x, \xi) = e^{-i \ev{x}{\xi}} V_{\phi_{A^{-1}}} \widehat{f}(\xi, -x) . 
	\end{equation}
Moreover, one verifies
	\begin{equation}
		\label{eq:BargmannSTFT}
		V_{\phi_A} f(A^{-1/2} x, A^{1/2} \xi) = (2 \pi)^{-\frac{d}{2}} e^{-\frac14 |z|^2} e^{- \frac{i}{2} \ev{x}{\xi}} \B_A f\left(\frac{\overline{z}}{\sqrt{2}}\right) , \qquad z = x + i\xi \in \C^d . 
	\end{equation}

\subsection{Proofs of Propositions \ref{p:t^2Est} and \ref{p:PreciseEstGenCase}}
We consider how to get bounds on the skewed Hermite coefficients of a function from its time-frequency decay.
We need several intermediate results.
First, we present a more precise multidimensional analog of \cite[Lemma 4.4]{N-T-V-HermExpSpFuncNearlyOptimalTFDecay}.

\begin{lemma}
\label{l:TFDecayToBargmann}
Let $\omega$ be a weight function satisfying $(\alpha)$ such that $\omega(t) = o(t^2)$ (such that $\omega(t) = t^2$).
If $f \in L^2(\R^d)$ satisfies
	\[
		|f(x)| \lesssim e^{-\frac{1}{2} \ev{Ax}{x} + \lambda \omega(|A^{1/2} x|)} \quad \text{and} \quad |\widehat{f}(\xi)| \lesssim e^{-\frac{1}{2} \ev{A^{-1}\xi}{\xi} + \lambda \omega(|A^{-1/2} \xi|)} , 
	\]
for some $\lambda > 0$ (for $0 < \lambda < 1/2$), then, for $z = x + i\xi \in \C^d$,
	\[
		|\B_A f(z)| \lesssim_\eta e^{\frac{1}{2} |x|^2 + \eta \omega\left(\frac{|\xi|}{\sqrt{2}}\right)} \quad \text{and} \quad |\B_A f(z)| \lesssim_\eta e^{\frac{1}{2} |\xi|^2 + \eta \omega\left(\frac{|x|}{\sqrt{2}}\right)} ,
		\qquad \eta > \ModOm \lambda . 
	\] 
\end{lemma}

\begin{proof}
For any $L > \ModOm$, we have
	\begin{align*}
		&|V_{\phi_A} f(x, \xi)| e^{\frac{1}{4} \ev{Ax}{x} - L \lambda \omega\left(\frac{|A^{1/2} x|}{2}\right)} \\
		& \qquad \qquad \qquad
		\lesssim \int_{\R^d} |f(t)|
			e^{-\frac{1}{2} \ev{A(t - x)}{t-x} + \frac{1}{4} \ev{Ax}{x}}
			e^{- L \lambda \omega\left(\frac{|A^{1/2} x|}{2}\right)}\, dt
		 \\
		& \qquad \qquad \qquad
		= \int_{\R^d} (|f(t)| e^{\frac{1}{2} \ev{At}{t}})
		e^{-\frac{1}{4} \ev{A(2t - x)}{2t - x} - L \lambda \omega\left(\frac{|A^{1/2} x|}{2}\right)}\,  dt
		\\
		& \qquad \qquad \qquad
		\lesssim \int_{\R^d} (|f(t)|
			e^{\frac{1}{2} \ev{At}{t} - \lambda \omega(|A^{1/2} t|)})
			e^{L \lambda \omega\left(\frac{|A^{1/2}(2t - x)|}{2}\right) - \frac{1}{4} |A^{1/2}(2t - x)|^{2}}\, dt
		\\
		& \qquad \qquad \qquad
		\lesssim 1 ,
	\end{align*}
where in the case $\omega(t) = t^2$ we used that $\Mod{t^2} = 2$ and $\lambda < 1/2$.
Then, by \eqref{eq:BargmannSTFT},
	\[ |\B_A f(z)| \lesssim |V_{\phi_A} f(\sqrt{2} A^{-1/2} x, -\sqrt{2} A^{1/2} \xi)| e^{\frac{1}{2} |x|^2 + \frac{1}{2} |\xi|^2}  \lesssim e^{\frac{1}{2} |\xi|^2 + L \lambda \omega\left(\frac{|x|}{\sqrt{2}}\right)} .  \]
The other bound follows similarly by \eqref{eq:STFTFourier}.
\end{proof}

Next, we consider how to handle bounds in the Fock space.
Here, we use several tools from complex analysis.
We start by establishing estimates for the one-dimensional case.

\begin{lemma}
\label{l:PLt^2Result}
Let $F$ be an entire function on $\C$ such that
	\[ |F(z)| \leq C e^{\frac12 x^2 + \lambda \xi^2} \quad \text{and} \quad |F(z)| \leq C e^{\frac12 \xi^2 + \lambda x^2} , \qquad z = x + i \xi \in \C , \]
for some $0 < \lambda < 1/2$ and $C > 0$, then, for some $C_\lambda > 0$ (depending only on $\lambda$),
	\begin{equation}
		\label{eq:PLt^2Result} 
		|F(z)| \leq C C_\lambda e^{ \sqrt{\frac{\lambda}{2}} |z|^2} , \qquad z \in \C . 
	\end{equation}
\end{lemma}

\begin{proof}
For any $\theta \in (0, \pi/2)$ and $r > 0$, we have
	\[ |F(r e^{i \varrho})| \leq C \exp\left[ \left( \lambda + \left\{ \frac{1}{2} - \lambda \right\} \sin^2 \left( \frac{\pi}{4} - \frac{\theta}{2} \right) \right) r^2 \right]  \]
for $\varrho \in \left[0, \frac{\pi}{4} - \frac{\theta}{2}\right) \cup \left(\frac{\pi}{4} + \frac{\theta}{2}, \frac{\pi}{2}\right]$.
Noting that $\sin^2 (\frac{\pi}{4} - \frac{\theta}{2}) = \frac12 - \frac12 \sin \theta$, and applying the Phragm\'{e}n-Lindel\"{o}f principle for $\omega(t) = t^2$ to the sector $\{ z \in \C \mid \arg z \in (\frac{\pi}{4} - \frac{\theta}{2}, \frac{\pi}{4} + \frac{\theta}{2}) \}$, it follows from Lemma \ref{l:ExamplesPL}$(i)$ (the index is a minimum) that for some $C_{\lambda, \theta} > 0$ (independent of $F$)
	\[ |F(z)| \leq C C_{\lambda, \theta} \exp\left[ \left( \frac12 \left\{ \frac12 + \lambda \right\} \sec \theta - \frac12 \left\{ \frac12 - \lambda \right\} \tan \theta \right) |z|^2 \right] , \qquad 0 \leq \arg z \leq \frac{\pi}{2} . \]
The minimal value for the coefficient is reached at 
	\[ \theta_0 = \sin^{-1}\left(\frac{\frac12 - \lambda}{\frac12 + \lambda}\right)  \]
giving
	\[ |F(z)| \leq C C_{\lambda, \theta_0} e^{ \sqrt{\frac{\lambda}{2}} |z|^2} , \qquad 0 \leq \arg z \leq \frac{\pi}{2} . \]
Repeating this for every quadrant of $\C$, we get, for $C_\lambda = C_{\lambda, \theta_0}$, the bound \eqref{eq:PLt^2Result}.
\end{proof}

\begin{lemma}
\label{l:PLomegaResult}
Let $\omega$ be a weight function satisfying $(\beta^*_2)$.
Let $F$ be an entire function on $\C$ such that, for some $\eta, \nu, C > 0$,
	\[ |F(z)| \leq C e^{\frac12 x^2 + \eta \omega(\nu |\xi|)} \quad \text{and} \quad |F(z)| \leq C e^{\frac12 \xi^2 + \eta \omega(\nu |x|)} , \qquad z = x + i \xi \in \C . \]
Then, for every $\zeta > \PL_{\frac{\pi}{2}}(\omega) \eta$ there exists some $C_\zeta > 0$ (depending only on $\omega$ and $\zeta$) such that
	\begin{equation}
		\label{eq:PLomegaResult} 
		|F(z)| \leq C C_\zeta e^{ \zeta \omega(\nu |z|)} , \qquad z \in \C . 
	\end{equation}
\end{lemma}

\begin{proof}
On the imaginary and real lines, we have 
	\[ |F(x)| \leq C e^{\eta \omega(\nu |x|)} \quad \text{and} \quad |F(i\xi)| \leq C e^{\eta \omega(\nu |\xi|)} , \qquad x, \xi \in \R . \]
Also, since $\omega(t) = o(t^2)$ by Lemma \ref{l:WeightFuncBasicProp}$(iii)$, we obtain
	\[ |F(z)| \lesssim_\varepsilon e^{\varepsilon |z|^2} , \qquad \text{for each } \varepsilon > 0, \] 
from Lemma \ref{l:PLt^2Result}.
By Lemma \ref{l:PLBounds}$(i)$, it follows $\PL_{\frac{\pi}{2}}(\omega) < \infty$ since $\omega$ satisfies $(\beta^*_2)$.
Therefore, applying the Phragm\'en-Lindel\"{o}f principle for $\omega$ to each quadrant of $\C$, we find \eqref{eq:PLomegaResult} for each $\zeta > \PL_{\frac{\pi}{2}}(\omega) \eta$.
\end{proof}

The crux of showing the multidimensional case lies in the following argument.
It is a more precise iteration of Proposition \ref{t:Homogeneous}.
Recall, $\mathbb{S}^{d-1}$ is the unit $d$-sphere in $\R^d$.

\begin{proposition}
\label{p:HomogeneousTrick}
Let $\omega$ be a weight function satisfying $(\alpha)$ and $(\delta)$.
Let $F$ be an entire function on $\C^d$.
Suppose that, for some $\zeta, \nu, C_F > 0$, we have,
	\[ |F_u(z)| := |F(u_1 z, \ldots, u_d z)| \leq C_F e^{\zeta \omega(\nu |z|)} , \qquad z \in \C , ~ u \in \mathbb{S}^{d-1} . \]
Then, 
	\begin{equation}
		\label{eq:BoundsOnDerivsAtZero}
		|F^{(\alpha)}(0)| \lesssim_r \alpha! e^{- \frac{1}{r} \varphi^*(r |\alpha|)} , \qquad 0 < r < \left( \ModKOm{\nu d} \zeta \right)^{-1} . 
	\end{equation}
\end{proposition}

\begin{proof}
	For any $u \in \mathbb{S}^{d-1}$, the function $F_u$ is entire on $\C$ with
		\[ 
			F_u(z) 
			= \sum_{\alpha \in \N^d} \frac{F^{(\alpha)}(0)}{\alpha!} \prod_{j = 1}^d (u_j z)^{\alpha_j}
			= \sum_{n = 0}^\infty \left( \sum_{|\alpha| = n} \frac{F^{(\alpha)}(0)}{\alpha!} u^\alpha \right) z^n .
		\]
	Therefore, for any $n \in \N$ and $u \in \mathbb{S}^{d-1}$, by the Cauchy inequalities,
		\begin{align*}
			\left| \sum_{|\alpha| = n} \frac{F^{(\alpha)}(0)}{\alpha!} u^\alpha \right|
			&= \frac{|F_u^{(n)}(0)|}{n!}
			\leq \inf_{t \geq \nu^{-1}} \left( \sup_{|z| = t} \frac{|F_u(z)|}{t^n} \right) \\
			&\leq C_F \inf_{t \geq \nu^{-1}} \exp\left( - [n \log t - \zeta \omega(\nu t) ] \right) \\
			&= C_F \nu^{n} e^{-\zeta \varphi^*\left(\frac{n}{\zeta}\right)} .
		\end{align*}
	Consequently, for a fixed $n \in \N$, the homogenous polynomial
		\[ P_n(x) := \frac{e^{\zeta \varphi^*\left(\frac{n}{\zeta}\right)}}{C_F \nu^n} \sum_{|\alpha| = n} \frac{F^{(\alpha)}(0)}{\alpha!} x^\alpha \]
	is bounded by $1$ on $\mathbb{S}^{d-1}$.
	By applying \cite[Theorem II]{K-BoundPolySevVar}, we find
		\[ |F^{(\alpha)}(0)| \leq C_F \alpha! \cdot (\nu d)^{|\alpha|} e^{-\zeta \varphi^*\left(\frac{|\alpha|}{\zeta}\right)} , \qquad \alpha \in \N^d . \]
	Then Lemma \ref{l:AddingSmallRegularity} yields \eqref{eq:BoundsOnDerivsAtZero}.
\end{proof}

\begin{proof}[Proof of Proposition \ref{t:Homogeneous}]
If $F$ is an entire function satisfying \eqref{eq:Homogeneous1}, we have, for some $\zeta > 0$ (resp. for all $\zeta > 0$), 
	\[ |F_u(s)| \lesssim e^{\zeta \omega(|s|)}, \qquad (s, u) \in \C \times \mathbb{S}^{d-1} . \]
By Proposition \ref{p:HomogeneousTrick}, the bound \eqref{eq:BoundsOnDerivsAtZero} holds for some $r > 0$ (resp. for every $r > 0$).
Now, by Corollary \ref{c:SeqSummable}, for any $0 < r' < r$,
	\[ |F(z)| \leq \sum_{\alpha \in \N^d} \frac{|F^{(\alpha)}(0)|}{\alpha!} |z|^{|\alpha|} \lesssim \sup_{\alpha \in \N^d} e^{|\alpha| \log |z| - \frac{1}{r'} \varphi^*(r' |\alpha|)} \leq e^{\frac{1}{r'} \omega(|z|)} . \]
Thus, $F$ satisfies \eqref{eq:Homogeneous2}.
\end{proof}

We may now complete the proofs of Propositions \ref{p:t^2Est} and \ref{p:PreciseEstGenCase}.

\begin{proof}[Proof of Proposition \ref{p:t^2Est}]
Suppose $f \in L^2(\R^d)$ is such that, for some $0 < \lambda < 1/2$,
	\[ |f(x)| \lesssim e^{- \left(\frac12 - \lambda\right)  \ev{Ax}{x}} \quad \text{and} \quad |\widehat{f}(\xi)| \lesssim e^{- \left(\frac12 - \lambda\right) \ev{A^{-1} \xi}{\xi}} . \]
Then, by Lemmas \ref{l:ExamplesAlpha} and \ref{l:TFDecayToBargmann}, we have for some $C > 0$ ($\Mod{t^2}$ is a minimum)
	\[ |\B_A f(z)| \leq C e^{\frac12 |x|^2 + \lambda |\xi|^2} \quad \text{and} \quad |\B_A f(z)| \leq C e^{\frac12 |\xi|^2 + \lambda |x|^2} , \qquad z = x + i\xi \in \C^d . \]
We define $F_u(z) = \B_A f(u_1 z, \ldots, u_d z)$, where $u \in \mathbb{S}^{d-1}$ and $z \in \C$.
Then also
	\[ |F_u(z)| \leq C e^{\frac12 |x|^2 + \lambda |\xi|^2} \quad \text{and} \quad |F_u(z)| \leq C e^{\frac12 |\xi|^2 + \lambda |x|^2} , \qquad z = x + i\xi \in \C , \]
for each $u \in \mathbb{S}^{d-1}$.
Applying Lemma \ref{l:PLt^2Result}, we get the bound
	\[ |F_u(z)| \leq C C_\lambda \exp\left[ \sqrt{\frac{\lambda}{2}} \cdot |z|^2 \right] , \qquad u \in \mathbb{S}^{d-1} , \, z \in \C . \]
For 
	\[ r = \frac{1}{d^2 \sqrt{\frac{\lambda}{2}}} , \]
we see by Proposition \ref{p:HomogeneousTrick} that ($\ModK{d}{t^2}$ is a minimum)
	\[ |\B_A^{(\alpha)}(0)| \lesssim \alpha! e^{-\frac{1}{r} \varphi^*(r |\alpha|)} , \qquad \alpha \in \N^d . \]
By \eqref{eq:CompSqrtFactAndYoung},
	\[
		\sqrt{\alpha!} e^{-\frac{1}{r} \varphi^*(r |\alpha|)}  
		\lesssim \left(\frac{(1 + \varepsilon) 2}{r}\right)^{\frac{|\alpha|}{2}}
		= \left((1 + \varepsilon) d^2 \sqrt{2 \lambda}\right)^{\frac{|\alpha|}{2}} ,
		\qquad \varepsilon > 0.
	\]
We obtain \eqref{eq:t^2EstHermite} from \eqref{eq:HermiteCoeffViaBargmann}.
\end{proof}

\begin{proof}[Proof of Proposition \ref{p:PreciseEstGenCase}]
Suppose $f \in L^2(\R^d)$ has the specified time-frequency decay for $\lambda > 0$.
We define $F_u(z) = \B_A f(u_1 z, \ldots, u_d z)$ for $u \in \mathbb{S}^{d-1}$ and $z \in \C$.
Using Lemmas \ref{l:TFDecayToBargmann} and \ref{l:PLomegaResult}, for any $\zeta > \ModOm \PL_{\frac{\pi}{2}}(\omega) \lambda$ there is some $C' > 0$ such that
	\[ |F_u(z)| \leq C' e^{\zeta \omega\left(\frac{|z|}{\sqrt{2}}\right)} , \qquad z \in \C , \, u \in \mathbb{S}^{d-1} . \]
Applying Proposition \ref{p:HomogeneousTrick} and using \eqref{eq:HermiteCoeffViaBargmann} give us
	\[ |H_A(f, \alpha)| \lesssim_r \sqrt{\alpha!} e^{-\frac{1}{r} \varphi^*(r |\alpha|)} , \qquad 0 < r < \frac{1}{\mathfrak{C}_{1,d}(\omega) \lambda} . \]
\end{proof}

\subsection{Proofs of Propositions \ref{p:HermiteToTFEstFort^2} and \ref{p:HermiteToTFEstGeneral}}
We shall now show how to conclude specified time-frequency decay from bounds on the skewed Hermite coefficients.
We need two intermediate results.

\begin{lemma}
\label{l:HermiteToBargmann}
Let $\omega$ be a weight function satisfying $(\alpha)$ and $(\delta)$.
If $f \in L^2(\R^d)$ is such that
	\begin{equation}
		\label{eq:BargmannToHermiteCoeff}
		|H_A(f, \alpha)| \lesssim \sqrt{\alpha!} e^{-\frac{1}{r} \varphi^*(r |\alpha|)} , 
	\end{equation}
for some $r > 0$, then,
	\begin{equation}
		\label{eq:BargmannToHermiteEntire}
		|\B_A f(z)| \lesssim_\lambda e^{\lambda \omega(|z|)} , \qquad \lambda > \frac{1}{r} .
	\end{equation}
\end{lemma}

\begin{proof}
By Corollary \ref{c:SeqSummable}, for any $s < r$, we have
	\[ \sum_{\alpha \in \N^d} \frac{|H_A(f, \alpha)|}{\sqrt{\alpha!}} e^{\frac{1}{s} \varphi^*(s |\alpha|)} < \infty . \]
Then, for $\lambda = 1/s$,
	\begin{align*} 
		|\B_A f(z)|
		&= \left| \B_A \left[ \sum_{\alpha \in \N^d} H_A(f, \alpha) h_{A, \alpha} \right] (z) \right| 
		\leq \sum_{\alpha \in \N^d} \frac{|H_A(f, \alpha)|}{\sqrt{\alpha!}} |z|^{|\alpha|} \\
		&\lesssim \sup_{\alpha \in \N^d} \exp[ |\alpha| \log |z| - \lambda \varphi^*(|\alpha| / \lambda) ]
		\leq e^{\lambda \omega(|z|)} .
	\end{align*}
\end{proof}

\begin{lemma}
\label{l:BargmannToTFDecay}
Let $\omega$ be a weight function satisfying $(\alpha)$ such that $\omega(t) = o(t^2)$ (such that $\omega(t) = t^2$).
If $F$ is an entire function for which
	\[ |F(z)| \lesssim e^{\lambda \omega(|z|)} \]
for some $\lambda > 0$ (for $0 < \lambda < 1/8$), then $F = \B_A f$ for some $f \in L^2(\R^d)$ such that
	\[ |f(x)| \lesssim e^{-\frac12 \ev{Ax}{x} + \eta \omega(|A^{1/2} x|)} \quad \text{and} \quad |\widehat{f}(\xi)| \lesssim e^{-\frac12 \ev{A^{-1} \xi}{\xi} + \eta \omega(|A^{-1/2} \xi|)} \]
for $\eta > \ModOm \ModKOm{\sqrt{2}} \lambda$.
\end{lemma}

\begin{proof}
As $F$ is an entire function with $|F(z)| \lesssim e^{c' |z|^2}$ for some $c' < 1/2$, there exists $f \in L^2(\R^d)$ such that $F = \B_A f$.
Indeed, for $A = I$ this was shown in \cite{B-HilbertSpAnalFuncAssocIntTrans}, while the general case then follows from a change of variables.
By Fourier inversion, we have
	\[ \pi^{-\frac{d}{4}} (\det A)^{1/4} e^{-\frac{1}{8} \ev{Ax}{x}} f\left(\frac{x}{2}\right) = f\left(\frac{x}{2}\right) \phi_A\left(\frac{x}{2} - x\right) = (2 \pi)^{-\frac{d}{4}} \int_{\R^d} V_{\phi_A} f(x, \xi) e^{\frac{i}{2} \ev{x}{\xi}} d\xi . \]
Let $L > \ModOm$ and $L' > \ModKOm{\sqrt{2}}$.
By \eqref{eq:BargmannSTFT}, we find
	\begin{align*} 
		|V_{\phi_A} f(x, \xi)| 
		&\lesssim e^{-\frac14 (|A^{1/2} x|^2 + |A^{-1/2} \xi|^2)} e^{\lambda \omega\left(\frac{|A^{1/2} x| +|A^{-1/2} \xi|}{\sqrt{2}}\right)} \\
		&\lesssim \exp\left[ -\frac14 |A^{1/2} x|^2 + L \lambda \omega\left(\frac{|A^{1/2} x|}{\sqrt{2}}\right) - \frac14 |A^{-1/2} \xi|^2 +  L \lambda \omega\left(\frac{|A^{-1/2} \xi|}{\sqrt{2}}\right) \right] .  
	\end{align*}
Since 
	\[ \exp\left[ - \frac14 |A^{-1/2} \xi|^2 + \lambda L \omega\left(\frac{|A^{-1/2} \xi|}{\sqrt{2}}\right) \right] \in L^1(\R^d) ,\] 
(with the additional assumption that $\lambda < 1 / 8$ when $\omega(t) = t^2$), we see that
	\[ |f(x)| \lesssim e^{-\frac12 \ev{Ax}{x} + L \lambda \omega(\sqrt{2} |A^{1/2} x|)} \lesssim e^{-\frac12 \ev{Ax}{x} + L L' \lambda \omega(|A^{1/2} x|)} . \]
The other bound follows similarly by \eqref{eq:STFTFourier}.
\end{proof}

\begin{proof}[Proof of Proposition \ref{p:HermiteToTFEstFort^2}]
For $r = \frac{2}{d h^2}$, we have, by \eqref{eq:CompSqrtFactAndYoung},
	\[
		h^{|\alpha|}
		= \frac{\sqrt{\alpha!}}{\sqrt{\alpha!}} h^{|\alpha|}
		\leq \sqrt{\alpha!} \frac{1}{\sqrt{|\alpha|!}} (d h^2)^{\frac{|\alpha|}{2}}
		\lesssim \sqrt{\alpha!} e^{-\frac{1}{r} \varphi^*(r |\alpha|)} .
	\]
Lemma \ref{l:HermiteToBargmann} shows
	\[ |\B_A f(z)| \lesssim_\lambda e^{\lambda |z|^2} , \qquad \lambda > \frac{d h^2}{2} . \]
Note that $d h^2 / 2 < 1/8$.
The proof is finished by applying Lemma \ref{l:BargmannToTFDecay}, and using Lemma \ref{l:ExamplesAlpha}$(i)$.
\end{proof}

\begin{proof}[Proof of Proposition \ref{p:HermiteToTFEstGeneral}]
Directly from Lemmas \ref{l:HermiteToBargmann} and \ref{l:BargmannToTFDecay}.
\end{proof}

\subsection{Proof of Theorem \ref{t:AandBNotInverses}}
\label{sec:ProofBNotInverseA}

We immediately proceed to the proof.

\begin{proof}[Proof of Theorem \ref{t:AandBNotInverses}]
For $d = 1$, the result coincides with Theorem \ref{t:LogCase} with $N = 0$.
We thus suppose $d \geq 2$.
Put $\Upsilon = B^{1/2} A B^{1/2}$.
Let $\{ \nu_j \mid 1 \leq j \leq d \}$ be an orthonormal basis of eigenvectors of $\Upsilon$, and let $\gamma_j > 0$, $1 \leq j \leq d$, be the corresponding eigenvalues.
We denote by $Q$ the orthogonal matrix for which $Q e_j = \nu_j$ and consider the diagonal matrix $\Gamma = Q^{T} \Upsilon Q$.
To any $f \in L^2(\R^d)$ we associate the pull-back $g(x) := f(B^{1/2} Q \Gamma^{-1/4} x)$.
Suppose $f$ satisfies \eqref{eq:TFDecayt^2WithB} for some $\lambda \geq 0$.
Then,
	\begin{multline} 
		\label{eq:PullBack1}
		|g(x)| 
		\lesssim e^{-\left(\frac12 - \lambda\right) \ev{\Upsilon Q \Gamma^{-1/4} x}{Q \Gamma^{-1/4} x}} \\
		= e^{-\left(\frac12 - \lambda\right) \sum_{j = 1}^d \gamma_j |\ev{Q \Gamma^{-1/4} x}{\nu_j}|^{2}} 
		= e^{-\left(\frac12 - \lambda\right) \sum_{j = 1}^d \sqrt{\gamma_j} |x_j|^{2}}  
	\end{multline}
and
	\begin{equation}
		\label{eq:PullBack2} 
		|\widehat{g}(\xi)| \lesssim |\widehat{f}(B^{-1/2} Q \Gamma^{1/4} \xi)| \lesssim e^{-\left(\frac12 - \lambda\right) \sum_{j=1}^d \sqrt{\gamma_j} |\xi|^2} . 
	\end{equation}
Conversely, for $g \in L^2(\R^d)$, we set $f(x) := g(\Gamma^{1/4} Q^{-1} B^{-1/2} x)$; then $f$ satisfies \eqref{eq:TFDecayt^2WithB} if $g$ satisfies both \eqref{eq:PullBack1} and \eqref{eq:PullBack2} for some $\lambda \geq 0$.
We now consider the two cases separately.

$(i)$ 
If $B^{-1} - A$ is not semi-positive definite, then neither is $I - M$, so that necessarily $\max_{1 \leq j \leq d} \gamma_j > 1$.
Without loss of generality, we may assume $\gamma_1 > 1$ is maximal.
Suppose first $f \in L^2(\R^d)$ satisfies \eqref{eq:TFDecayt^2WithB} for $\lambda < \frac12(1 - \gamma_1^{-\frac12})$.
We apply a similar trick as in \cite{B-R-UncertaintyPrinciplesLieGroups}.
The function $g_{x'}(x_1) = g(x_1, x')$ is in $L^2(\R)$ for all $x' \in \R^{d-1}$, and its Fourier transform is given by
	\[ \widehat{g}_{x'}(\xi_1) = \frac{1}{(2\pi)^{d-1}} \int_{\R^{d-1}} \widehat{g}(\xi_1, \xi') e^{i \ev{x'}{\xi'}} d\xi' . \]
We find the bounds
	\[ |g_{x'}(x_1)| \lesssim_{x'} e^{-\left(\frac12 - \lambda\right) \sqrt{\gamma_1} |x_1|^2} \quad \text{and} \quad |\widehat{g}_{x'}(\xi_1)| \lesssim_{x'} e^{-\left(\frac12 - \lambda\right) \sqrt{\gamma_1} |\xi_1|^2} . \]
Now, by our assumption on $\lambda$, we have
	\[ \left(\frac12 - \lambda\right) \sqrt{\gamma_1} > \frac12 . \]
The classical Hardy uncertainty principle \cite{Hardy33} (see also Theorem \ref{t:LogCase}) gives $g_{x'} \equiv 0$ for all $x' \in \R^{d-1}$.
Therefore, also $f \equiv 0$.

On the other hand, the function
	\[ g(x) = \exp \left[ -\frac12 \sum_{j=1}^d \sqrt{ \frac{\gamma_j}{\gamma_1}} |x_j|^2 \right] \] 
satisfies \eqref{eq:PullBack1} and \eqref{eq:PullBack2} when $\lambda = \frac12(1 - \gamma_1^{-\frac12})$.
Therefore, there is a non-trivial $f$ satisfying \eqref{eq:TFDecayt^2WithB} for $\lambda = \frac12(1 - \gamma_1^{-\frac12})$.

$(ii)$ 
If $B^{-1} - A$ is semi-positive definite, so is $I - \Upsilon$, therefore $\max_{1 \leq j \leq d} \gamma_j \leq 1$.
However, since $I - \Upsilon$ is non-trivial, it also follows that $\min_{1 \leq j \leq d} \gamma_j < 1$.
As $\gamma_j \leq 1$ for each $j \in \{1, \ldots, d\}$, it follows that $g(x) = \phi_I(x) = e^{-\frac12 |x|^2}$ satisfies \eqref{eq:PullBack1} and \eqref{eq:PullBack2} for $\lambda = 0$.
Consequently, $f$ satisfies \eqref{eq:OptimalWhenEigenvalue<1}.
Consider the vectors $\kappa_j = \sqrt{\gamma_j} A^{-1/2} B^{-1/2} \nu_j$, $1 \leq j \leq d$, which form another orthonormal basis.
Then
	\begin{align*}
		f(A^{-1/2} x) 
		&= \phi_I(\Gamma^{1/4} Q^{-1} B^{-1/2} A^{-1/2} x) \\
		&= \exp\left[ -\frac12 \sum_{j = 1}^d |\ev{\Gamma^{1/4} Q^{-1} B^{-1/2} A^{-1/2} x}{e_j}|^2 \right] 
		= \exp\left[ -\frac12 \sum_{j = 1}^d \frac{1}{\sqrt{\gamma_j}} |\ev{x}{\kappa_j}|^2 \right] \\
		&= \exp\left[ -\frac12 |x|^2 - \frac12 \sum_{j = 1}^d \left( \frac{1}{\gamma_j} - 1 \right) |\ev{x}{\kappa_j}|^2 \right]. \\
	\end{align*}
Since $\gamma_j < 1$ for at least one $1 \leq j \leq d$, we see by the classical Hardy uncertainty principle \cite{Hardy33} (see also Theorem \ref{t:LogCase}) that $f(A^{-1/2} x)$ cannot satisfy \eqref{eq:TFDecayt^2} for each $0 < \lambda < 1/2$.
Therefore, by Proposition \ref{p:HermiteToTFEstFort^2}, there is some $h > 0$ such that
	\[ \sup_{\alpha \in \N^d} \frac{|H_I(f(A^{-1/2} \cdot), \alpha)|}{h^{|\alpha|}} = \infty . \]
By observing
	\[ H_{A}(f, \alpha) = (\det A)^{-\frac14} H_{I}(f(A^{-1/2} \cdot), \alpha) , \qquad \alpha \in \N^d , \]
we conclude that the subspace 
	\[ \mathcal{V} = \{ f \in L^2(\R^d) \mid f \text{ satisfies \eqref{eq:OptimalWhenEigenvalue<1} and \eqref{eq:NoQuantify} for some } h > 0 \} \] 
is non-trivial.

Suppose now $B^{-1} - A$ is positive definite.
In particular, $\gamma_j < 1$ for each $1 \leq j \leq d$.
Then, for any polynomial $p$, the function $g(x) = p(x) \phi_I(x)$ satisfies \eqref{eq:PullBack1} and \eqref{eq:PullBack2} for $\lambda = 0$.
A similar argument as before shows that the associated $f$ belongs to $\mathcal{V}$. 
This clearly shows that $\mathcal{V}$ is dense in $L^2(\R^d)$.

Finally, assume $B^{-1} - A$ is not positive definite.
Then $\gamma_j = 1$ for at least one $1 \leq j \leq d$. Without loss of generality, $\gamma_1 = 1$.
Let $g$ be any element of $L^2(\R^d)$ satisfying \eqref{eq:PullBack1} and \eqref{eq:PullBack2} for $\lambda = 0$.
Fix $x' \in \R^{d-1}$ and put $g_{x'}(x_1) = g(x_1, x')$.
A similar argument as in $(i)$ shows that 
	\[ |g_{x'}(x_1)| \lesssim_{x'} e^{-\frac12 |x_1|^2} \quad \text{and} \quad |\widehat{g}_{x'}(\xi_1)| \lesssim_{x'} e^{-\frac12 |\xi_1|^2} . \]
By the classical Hardy uncertainty principle \cite{Hardy33} (see also Theorem \ref{t:LogCase}), we see that $g_{x'}(x_1) = c_{x'} e^{-\frac12 |x_1|^2}$ for $c_{x'} \in \C$ depending only on $x'$.
This implies $H_I(g, \alpha) = 0$ whenever $\alpha_1 > 0$.
From here, we conclude that $\mathcal{V}$ cannot be dense in $L^2(\R^d)$.
\end{proof}

\section{Quantitative uncertainty principles for coordinate-wise time-frequency decay}
\label{sec:Coordinates}

We now turn our attention to functions satisyfing time-frequency Gaussian decay estimates a priori just in the $A^{1/2}$- and $A^{-1/2}$-skewed directions.
In the limit case of $\omega(t) = t^2$, we will recover Proposition \ref{p:t^2Est}, but with a higher loss of regularity.
For smaller weight functions, we will introduce the property of \emph{quadratic interpolation (of dimension $d$)}.
This property turns out to completely determine whether or not \eqref{eq:NearlyOptimalTFDecaySpecifiedHermite} can be obtained from coordinate-wise time-frequency decay bounds.
For the remainder of this section, $A \in \GL_+(\R^d)$.
We write 
	\[ v_j = A^{-1/2} e_j \quad \text{and} \quad w_j = A^{1/2} e_j , \qquad j \in \{1, \ldots, d\} , \] 
where $\{ e_j \mid 1 \leq j \leq d \}$ denotes the standard Euclidean basis of $\R^d$.
Then both $\{v_j \mid 1 \leq j \leq d\}$ and $\{w_j \mid 1 \leq j \leq d\}$ form bases of $\R^d$ and $w_j = A v_j$ for $j \in \{1, \ldots, d\}$.

\begin{proposition}
\label{p:t^2EstCoordinates}
If $f \in L^2(\R^d)$ satisfies
	\[
		|f(x)| \lesssim e^{- \left(\frac{1}{2} - \lambda\right) |\ev{A x}{v_j}|^2} \quad \text{and} \quad |\widehat{f}(\xi)| \lesssim e^{- \left(\frac{1}{2} - \lambda\right) |\ev{A^{-1} \xi}{w_j}|^2} , \qquad 1 \leq j \leq d ,
	\]
for some $0 < \lambda < 1 / 2$, then,
	\begin{equation}
		\label{eq:t^2EstHermiteCoordinates}
		|H_A(f, \alpha)| \lesssim_h h^{|\alpha|} , \qquad h > \left(2 \lambda \right)^{\frac{1}{4d}} .
	\end{equation}
\end{proposition}

The proof of the previous result will be given in Section \ref{sec:ProofsCoordinateWiseSuff}.
It is now straightforward to see that Theorem \ref{t:CharNearlyOptimalTFDecay} is a direct consequence of it.

\begin{proof}[Proof of Theorem \ref{t:CharNearlyOptimalTFDecay}]
The implication $(iii) \Longrightarrow (ii)$ was proven in Proposition \ref{p:HermiteToTFEstFort^2}, while $(ii) \Longrightarrow (i)$ is immediate.
Proposition \ref{p:t^2EstCoordinates} then shows $(i) \Longrightarrow (iii)$.
\end{proof}

For general weight functions, we have the following estimate.
Its proof will also be given in Section \ref{sec:ProofsCoordinateWiseSuff}.

\begin{proposition}
\label{p:PreciseEstGenCaseCoordinates}
Let $\omega$ be a weight function satisfying $(\beta^*_2)$ and $(\delta)$.
If $f \in L^2(\R^d)$ satisfies, for $1 \leq j \leq d$,
	\begin{equation}
		\label{eq:TFDecayCoordinates}
		|f(x)| \lesssim e^{-\frac{1}{2} |\ev{A x}{v_j}|^2 + \lambda \omega(|\ev{A x}{v_j}|)} \quad \text{and} \quad |\widehat{f}(\xi)| \lesssim e^{-\frac{1}{2} |\ev{A^{-1} \xi}{w_j}|^2 + \lambda \omega(|\ev{A^{-1} \xi}{w_j}|)} , 
	\end{equation}
for some $\lambda > 0$, then,
	\begin{equation}
		\label{eq:HermiteDecayCoordinates} 
		|H_A(f, \alpha)| \lesssim_{r} \left(\alpha!\right)^{\frac{1}{2d}} e^{-\frac{1}{r} \varphi^*\left(\frac{r |\alpha|}{d}\right)} , \qquad 0 < r < (\mathfrak{C}_{2,d}(\omega) \lambda)^{-1} ,
	\end{equation}
with 
	\[ \mathfrak{C}_{2, d}(\omega) := \ModOm \ModKOm{\sqrt{d}} \PL_{\frac{\pi}{2}}(\omega)  . \]
\end{proposition}

Observe that the estimate \eqref{eq:HermiteDecayCoordinates} is in general significantly weaker than \eqref{eq:NearlyOptimalTFDecaySpecifiedHermite}. 
In order to address when a bound of type \eqref{eq:NearlyOptimalTFDecaySpecifiedHermite} could be obtained from \eqref{eq:HermiteDecayCoordinates}, we now formally introduce the property of quadratic interpolation for weight functions.
A thorough analysis of this property will be carried out in Section \ref{sec:GaussInterpol}.

\begin{definition}
Let $\omega$ be a weight function satisfying $(\delta)$ and $\omega(t) = o(t^2)$.
We say $\omega$ satisfies \emph{quadratic interpolation (of dimension $d \geq 2$)} if for some $L > 0$ there exists a $C > 0$ such that
	\begin{equation}
		\label{eq:GaussianInterpolation}
		\forall t \geq 0 ~ \exists \sigma \in [1, d] : \quad t^{2 \frac{d - \sigma}{d - 1}} + \omega\left( t^\sigma \right) \leq L \omega(t) + C . 
	\end{equation}
\medskip \indent
For any $\mu > 0$, we write
	\[ \GI_d(\omega, \mu) = \inf\left\{ L \in [1, \infty] \mid \exists C > 0 \, : \, \inf_{s \geq 0} \left[ \frac{d - 1}{2d} s^{\frac{2}{d - 1}} + \frac{1}{\mu} \omega\left( \frac{t^d}{s} \right) \right] \leq \frac{L}{\mu} \omega(t) + C \right\} . \]
Note that $\GI_d(\omega, \mu)$ is increasing in $\mu$.
We will see (Lemma \ref{l:GIViaCoeff}) that $\omega$ satisfies quadratic interpolation of dimension $d$ if and only if $\GI_d(\omega, \mu) < \infty$ for any/all $\mu > 0$.
\end{definition}

If the weight function satisfies quadratic interpolation, we recover Proposition \ref{p:PreciseEstGenCase} with a higher loss of regularity, as stated in the following result.
We provide a proof of it in Section \ref{sec:Counterexample}.

\begin{theorem}
\label{t:CoordinateWiseSucces}
Let $d \geq 2$.
Let $\omega$ be a weight function satisfying $(\beta^*_2)$, $(\delta)$, and quadratic interpolation of dimension $d$.
If $f \in L^2(\R^d)$ satisfies \eqref{eq:TFDecayCoordinates} for some $\lambda > 0$, then,
	\begin{equation}
		\label{eq:NearlyOptimalTFDecaySpecifiedHermite2}
		|H_A(f, \alpha)| \lesssim_r \sqrt{\alpha!} e^{-\frac{1}{r} \varphi^*(r |\alpha|)} , \qquad 0 < r < [\mathfrak{C}_{3, d}(\omega, \lambda) \lambda]^{-1} ,
	\end{equation}
where
	\[ \mathfrak{C}_{3, d}(\omega, \lambda) = \ModKOm{d^{\frac{d-1}{2d}}} \mathfrak{C}_{2, d}(\omega) \GI_d(\omega, (\mathfrak{C}_{2, d}(\omega) \lambda)^{-1}) .  \]
\end{theorem}

We point out that Theorem \ref{t:LogCase} follows at once from Theorem \ref{t:CoordinateWiseSucces}.
In fact, we shall see in Section \ref{sec:Counterexample} that the following more general result holds true.

\begin{theorem}
\label{t:LogCaseGeneral}
Let $A, B \in \GL_+(\R^d)$.
Let $f \in L^2(\R^d)$ be such that, for some $N \in \N$ and any $j \in \{1, \ldots, d\}$,
	\begin{equation} 
		\label{eq:PolynomialDecayCoordinates} 
		|f(x)| \lesssim (1 + |\ev{Ax}{v_j}|)^N e^{-\frac{1}{2} |\ev{A x}{v_j}|^2} \text{ and }  |\widehat{f}(\xi)| \lesssim (1 + |\ev{B \xi}{w_j}|)^N e^{-\frac{1}{2} |\ev{B \xi}{w_j}|^2} .
	\end{equation}
Then the following are true:
	\begin{itemize}
		\item[$(i)$] If $B^{-1} - A$ is not semi-positive definite, then $f$ is null;
		\item[$(ii)$] If $B = A^{-1}$, then $f(x) = p(A^{1/2} x) e^{-\frac{1}{2} |\ev{Ax}{x}|^2}$ for some polynomial $p$ of degree at most $N$ in each variable;
		\item[$(iii)$] If $B^{-1} - A$ is semi-positive definite, the subspace
			\[ \mathcal{W} = \{ f \in L^2(\R^d) \mid f \text{ satisfies \eqref{eq:PolynomialDecayCoordinates} and \eqref{eq:NoQuantify} for some } h > 0 \} \] 
		is non-trivial. The subspace $\mathcal{W}$ is dense in $L^2(\R^d)$ if and only if $B^{-1} - A$ is positive definite.
	\end{itemize}
\end{theorem}

\begin{remark}
It should be noticed that the third statement of Proposition 3.3 in the preprint version\footnote{Available at \href{https://doi.org/10.48550/arXiv.math/0102111}{arXiv:math/0102111}.} of \cite[Proposition 3.4]{BDJ2003} is incorrect. Indeed, they claim there that if $B^{-1} - A$ is not semi-negative definite, then there is a dense subspace of functions $f$ satisfying \eqref{eq:PolynomialDecayCoordinates}. However, if $A$ and $B$ are such that $B^{-1} - A$ is either not positive definite or neither semi-positive definite nor semi-negative definite, Theorem \ref{t:LogCaseGeneral} above yields that this is not possible.
\end{remark}

The scope of Theorem \ref{t:CoordinateWiseSucces} is significantly broader than that of Theorem \ref{t:LogCaseGeneral}.
We illustrate this with the weight functions $\omega(t) = \log_+^{1 + a}(t)$ for $a > 0$.
The next result can be seen as a natural extension of Theorem \ref{t:LogCase}.
Its proof will be given in Section \ref{sec:Counterexample} as well.

\begin{theorem}
\label{t:CoordinatesLog}
Let $a > 0$.
If $f \in L^2(\R^d)$ is such that, for $j \in \{1, \ldots, d\}$,
	\[ |f(x)| \lesssim e^{-\frac{1}{2} |\ev{A x}{v_j}|^2 + \lambda \log_+^{1 + a}(|\ev{A x}{v_j}|)} \quad \text{and} \quad |\widehat{f}(\xi)| \lesssim e^{-\frac{1}{2} |\ev{A^{-1} \xi}{w_j}|^2 + \lambda \log_+^{1 + a}(|\ev{A^{-1} \xi}{w_j}|)} , \]
for some $\lambda > 0$, then, 
	\[ |H_A(f, \alpha)| \lesssim e^{- \left(\frac{r  |\alpha|}{d}\right)^{\frac{a + 1}{a}}} , \qquad 0 < r <  \frac{a^{\frac{a}{1+a}} }{1+a} \lambda^{-\frac{1}{1+a}} . \]
\end{theorem}

When the weight function does not satisfy quadratic interpolation, the decay of \eqref{eq:HermiteDecayCoordinates} is significantly slower than that of \eqref{eq:NearlyOptimalTFDecaySpecifiedHermite2}.
For such weight functions, we may then construct elements of $L^2(\R^d)$ satisfying the coordinate-wise time-frequency decay, but not \eqref{eq:NearlyOptimalTFDecaySpecifiedHermite}.
We demonstrate this in two ways.
The proof of the next theorem will be given in Section \ref{sec:Counterexample}.

\begin{theorem}
\label{t:CounterExample}
Let $d \geq 2$.
Let $\omega$ be a weight function satisfying $(\alpha)$ and $(\delta)$ but not quadratic interpolation of dimension $d$.

\begin{itemize}
\item[$(i)$] For any $\lambda > 0$, there exists $f \in L^2(\R^d)$ such that
	\begin{equation} 
		\label{eq:CounterExample}		
		\sup_{\alpha \in \N^d} \frac{|H_A(f, \alpha)|}{\sqrt{\alpha!} e^{-\frac{1}{\varepsilon} \varphi^*(\varepsilon |\alpha|)}} = \infty 
	\end{equation}
for every $\varepsilon > 0$, but such that $f$ satisfies \eqref{eq:TFDecayCoordinates} for $\lambda$ and, for some $r > 0$,
	\[ 
		\limsup_{|\alpha| \to \infty} \frac{|H_A(f, \alpha)|}{\left(\alpha!\right)^{\frac{1}{2d}} e^{-\frac{1}{r} \varphi^*\left(\frac{r |\alpha|}{d}\right)}} > 0 .
	\]
	
\item[$(ii)$] There exist $f \in L^2(\R^d)$ for which \eqref{eq:CounterExample} holds for every $\varepsilon > 0$, but for which \eqref{eq:TFDecayCoordinates} is satisfied for any $\lambda > 0$.
\end{itemize}
\end{theorem}

The weight functions $\omega(t) = t^a$ with $0 < a < 2$ satisfy the conditions of Theorem \ref{t:CounterExample}, see Corollary \ref{c:t^aNotQuadraticInterpol}.
We then find the following optimality result.
Its proof can be found in Section \ref{sec:Counterexample}.

\begin{theorem}
\label{t:Countexamplet^a}
Let $d \geq 2$ and $0 < a < 2$.

\begin{itemize}
\item[$(i)$] If $f \in L^2(\R^d)$ is such that, for $1 \leq j \leq d$,
	\begin{equation}
		\label{eq:Coordinatewiset^a} 
		|f(x)| \lesssim e^{-\frac{1}{2} |\ev{A x}{v_j}|^2 + \lambda |\ev{A x}{v_j}|^a} \quad \text{and} \quad |\widehat{f}(\xi)| \lesssim e^{-\frac{1}{2} |\ev{A^{-1} \xi}{w_j}|^2 + \lambda |\ev{A^{-1} \xi}{w_j}|^a} , 
	\end{equation}
for some $\lambda > 0$, then, 
	\[ |H_A(f, \alpha)| \lesssim h^{|\alpha|} (\alpha !)^{- \frac{1}{d} \left( \frac{1}{a} - \frac12 \right)} , \qquad h > d^{\frac{1}{d} \left( \frac{1}{a} + \frac{1}{2} \right)} \left( a\max\{ 2^{a - 1}, 1 \} \sec\left(\frac{a \pi}{4}\right) \lambda \right)^{\frac{1}{ad}}. \]
	
\item[$(ii)$] For any $\lambda > 0$, there exists $f \in L^2(\R^d)$ satisfying \eqref{eq:Coordinatewiset^a} for $\lambda$ such that, for some $h > 0$,
	\[ \limsup_{|\alpha| \to \infty} \frac{|H_A(f, \alpha)|}{h^\alpha (\alpha !)^{- \frac{1}{d} \left( \frac{1}{a} - \frac12 \right)}} > 0 . \]  
	
\item[$(iii)$] There exists $f \in L^2(\R^d)$ satisfying \eqref{eq:Coordinatewiset^a} for any $\lambda > 0$ and for which
	\begin{equation}
		\label{eq:t^aFailOtherPowers} 
		\sup_{\alpha \in \N^d} \frac{|H_A(f, \alpha)|}{h^\alpha (\alpha!)^{-b}} = \infty , \qquad \text{for all } b > \frac{1}{d} \left( \frac{1}{a} - \frac{1}{2} \right) \text{ and } h > 0 . 
	\end{equation}
\end{itemize}
\end{theorem}	

\subsection{Proofs of Propositions \ref{p:t^2EstCoordinates} and \ref{p:PreciseEstGenCaseCoordinates}}
\label{sec:ProofsCoordinateWiseSuff}

We again translate the time-frequency decay into bounds on the skewed Bargmann transform.
We find the following coordinate-wise version of Lemma \ref{l:TFDecayToBargmann}.

\begin{lemma}
\label{l:TFDecayToBargmannCoordinates}
Let $\omega$ be a weight function satisfying $(\alpha)$ such that $\omega(t) = o(t^2)$ (such that $\omega(t) = t^2$).
If $f \in L^2(\R^d)$ satisfies \eqref{eq:TFDecayCoordinates} for some $\lambda > 0$ (for $0 < \lambda < 1/2$), then, for $z = x + i\xi \in \C^d$ and $j \in \{1, \ldots, d\}$,
	\[
		|\B_A f(z)| \lesssim e^{\frac{1}{2} |x_j|^2 + \eta \omega\left(\frac{|\xi_j|}{\sqrt{2}}\right) + \frac{1}{2} (|z|^2 - |z_j|^2)} 
		\quad \text{and} \quad
		|\B_A f(z)| \lesssim e^{\frac{1}{2} |\xi_j|^2 + \eta \omega\left(\frac{|x_j|}{\sqrt{2}}\right) + \frac{1}{2} (|z|^2 - |z_j|^2)} , 
	\] 
for any $\eta > \ModOm \lambda$.
\end{lemma}

\begin{proof}
We demonstrate the bounds for $j = d$. 
The others are shown similarly.
By our assumptions, we have
	\[ |f(A^{-1/2} t)| \lesssim e^{-\frac12 |t_d|^2 + \lambda \omega(|t_d|)} . \]
For any $L > \ModOm$, 
	\begin{align*}
		&|V_{\phi_A} f(A^{-1/2} x, \xi)| e^{\frac{1}{4} |x_d|^{2} - L \lambda \omega\left(\frac{|x_d|}{2}\right)} \\
		&\lesssim \int_{\R^{d-1}} e^{-\frac{1}{2} \sum_{k < d} |t_k - x_k|^2} 
			\left( \int_{-\infty}^{\infty} |f(A^{-1/2} t)|
			e^{-\frac{1}{2} |t_d - x_d|^{2} + \frac{1}{4} |x_d|^{2} - L \lambda \omega\left(\frac{|x_d|}{2}\right)}
			dt_d \right) \, dt_{d- 1} \cdots dt_1
		 \\
		 &\lesssim \int_{\R^{d-1}} e^{-\frac{1}{2} \sum_{k < d} |t_k - \ev{x}{v_k}|^2} \left( \int_{-\infty}^{\infty} e^{L \lambda \omega\left(\frac{|2t_d - x_d|}{2}\right) - \frac{1}{4} |2t_d - x_d|^{2}} \,  dt_d \right) \, dt_{d- 1} \cdots dt_1
		 \\
		&\lesssim 1 ,
	\end{align*}
where in the case $\omega(t) = t^2$ we used that $\lambda < 1/2$.
Then \eqref{eq:BargmannSTFT} implies
	\[
		|\B_A f(z)| 
		\lesssim |V_{\phi_A} f(\sqrt{2} A^{-1/2} x, -\sqrt{2} A^{1/2} \xi)| e^{\frac{1}{2} |x|^2 + \frac{1}{2} |\xi|^2}  
		\lesssim e^{\frac{1}{2} |\xi_d|^2 + L \lambda \omega\left(\frac{|x_d|}{\sqrt{2}}\right) + \frac12 (|z|^2 - |z_j|^2)} .  
	\]
The other bound follows similarly by \eqref{eq:STFTFourier}.
\end{proof}

\begin{proof}[Proof of Proposition \ref{p:t^2EstCoordinates}]
Set $F = \B_A f$. Then, by Lemmas \ref{l:ExamplesAlpha}$(i)$ (the indices are minima) and \ref{l:TFDecayToBargmannCoordinates}, for $j \in \{1, \ldots, d\}$,
\[
	|F(z)| \leq C e^{\frac12 |x_j|^2 + \lambda |\xi_j|^2 + \frac12 (|z|^2 - |z_j|^2)} 
	\quad \text{and} \quad
	|F(z)| \leq C e^{\frac12 |\xi_j|^2 + \lambda |x_j|^2 + \frac12 (|z|^2 - |z_j|^2)}  ,
\]
for a certain $C > 0$.
Fix some $j \in \{1, \ldots, d\}$ and $z_k \in \C$, $k \neq j$.
Applying Lemma \ref{l:PLt^2Result} to the entire function $F_j(z) = F(z_1, \ldots, z_{j-1}, z, z_{j+1}, \ldots, z_d)$, we find
	\[ |F(z)| \leq C C_\lambda e^{\sqrt{\frac{\lambda}{2}} |z_j|^2 + \frac12 (|z|^2 - |z_j|^2)} , \qquad z \in \C^d , \]
for some $C_\lambda > 0$. 
Then, using the Cauchy inequalities and \eqref{eq:CompSqrtFactAndYoung}, we find, for $\alpha \in \N^d$ and any $\varepsilon > 0$,
	\begin{align*}
		\frac{|F^{(\alpha)}(0)|}{\alpha!} 
		&\leq \left( \inf_{t > 0} \frac{e^{\sqrt{\frac{\lambda}{2}} t^2}}{t^{\alpha_j}}  \right) \left( \prod_{k \neq j} \inf_{t > 0} \frac{e^{\frac12 t^2}}{t^{\alpha_k}} \right)
		\lesssim (1 + \varepsilon)^{|\alpha|} \left(2 \lambda\right)^{\frac{\alpha_j}{4}} (\alpha!)^{-\frac12} .
	\end{align*}
By \eqref{eq:HermiteCoeffViaBargmann} and the assumption $\lambda < 1/ 2$, we see that
	\[ |H_A(f, \alpha)| \leq \min_{1 \leq j \leq d} (1 + \varepsilon)^{|\alpha|} \left(2 \lambda \right)^{\frac{\alpha_j}{4}} \leq (1 + \varepsilon)^{|\alpha|} \left(2 \lambda \right)^{\frac{|\alpha|}{4d}} . \]
As $\varepsilon > 0$ was arbitrary, this allows us to conclude that \eqref{eq:t^2EstHermiteCoordinates} holds.
\end{proof}

\begin{proof}[Proof of Proposition \ref{p:PreciseEstGenCaseCoordinates}]
Put $F = \B_A f$. Similarly to the proof of Proposition \ref{p:t^2EstCoordinates}, we find, for $j \in \{1, \ldots, d\}$,
	\[ |F(z)| \lesssim e^{\eta \omega\left(\frac{|z_j|}{\sqrt{2}}\right) + \frac12(|z|^2 - |z_j|^2)} , \qquad \eta > \ModOm \PL_{\frac{\pi}{2}}(\omega) \lambda . \]
Using the Cauchy inequalities and \eqref{eq:CompSqrtFactAndYoung}, we get, for any $\varepsilon > 0$,
	\[ \frac{|F^{(\alpha)}(0)|}{\alpha!} \lesssim 2^{-\frac{|\alpha_j|}{2}} e^{- \eta \varphi^*\left(\frac{|\alpha_j|}{\eta}\right)} \prod_{k \neq j} (1 + \varepsilon)^{\alpha_k} (\alpha_k!)^{-\frac12} , \qquad j \in \{1, \ldots, d\} . \]
Hence, by \eqref{eq:HermiteCoeffViaBargmann}, Lemma \ref{l:AddingSmallRegularity}, and Corollary \ref{c:DecayHermiteBounds} yield, for any $1 < h' < h$,
	\begin{align*} 
		|H_A(f, \alpha)| 
		&\lesssim (h')^{\frac{|\alpha|}{d}} \min_{j \in \{1, \ldots, d\}} 2^{-\frac{|\alpha_j|}{2}} \sqrt{\alpha_j!} e^{- \eta \varphi^*\left(\frac{|\alpha_j|}{\eta}\right)} 
		\lesssim \left(h' \right)^{\frac{|\alpha|}{d}} \left(\left\lceil\frac{|\alpha|}{d}\right\rceil!\right)^{\frac12} e^{- \eta \varphi^*\left(\frac{|\alpha|}{d \eta}\right)} \\
		&\lesssim \left( h^{|\alpha|} \sqrt{|\alpha|!} \right)^{\frac{1}{d}} e^{-\eta \varphi^*\left(\frac{|\alpha|}{d \eta}\right)} 
		\lesssim  \left(\alpha!\right)^{\frac{1}{2d}} e^{ - \frac{1}{r} \varphi^*\left(\frac{r |\alpha|}{d}\right)} , 
	\end{align*}
for $0 < r < (\ModKOm{h} \ModKOm{\sqrt{d}} \eta)^{-1}$.
As $h$ can be chosen arbitrarily close to $1$, we conclude \eqref{eq:HermiteDecayCoordinates} from \eqref{eq:alpha->0}.
\end{proof}

\subsection{Quadratic interpolation}
\label{sec:GaussInterpol}

We consider the property of quadratic interpolation and the associated index $\GI_d(\omega, \mu)$ in more detail.
Throughout this subsection, we will assume $d \geq 2$ and $\omega$ is a weight function satisfying $(\delta)$ and $\omega(t) = o(t^2)$.

One verifies the following bounds.

\begin{lemma}
\label{l:GIUpperBound}
We have
	\begin{equation*}
		\label{eq:GIUpperBound} 
		\GI_d(\omega, \mu_0) \leq \max\{1, \frac{\mu_0}{\mu_1}\} \GI_d(\omega, \mu_1) , \qquad \mu_0, \mu_1 > 0 . 
	\end{equation*}
\end{lemma}

Quadratic interpolation is equivalent to the indices $\GI_d(\omega, \mu)$ being finite.

\begin{lemma}
\label{l:GIViaCoeff}
The weight function $\omega$ satisfies quadratic interpolation of dimension $d$ if and only if $\GI_d(\omega, \mu) < \infty$ for some/all $\mu > 0$.
\end{lemma}

\begin{proof}
Suppose \eqref{eq:GaussianInterpolation} holds for certain $L \geq 1$ and $C > 0$.
For $t$ large enough we find
	\begin{align*}
		\inf_{s \geq 0} \left[ \frac{d - 1}{2d} s^{\frac{2}{d-1}} + \omega\left(\frac{t^d}{s}\right) \right]
		&= \inf_{\sigma \in \R} \left[ \frac{d - 1}{2d} t^{2 \frac{d - \sigma}{d-1}} + \omega\left(t^\sigma\right) \right]
		\leq \inf_{1 \leq \sigma \leq d} \left[ t^{2 \frac{d - \sigma}{d-1}} + \omega\left(t^\sigma\right) \right] + 1 \\
		&\leq L \omega(t) + C + 1 .
	\end{align*}
This shows $\GI_d(\omega, 1) < \infty$.
Hence, by Lemma \ref{l:GIUpperBound}, we have $\GI_d(\omega, \mu) < \infty$ for any $\mu > 0$.	

Conversely, suppose $\GI_d(\omega, \mu) < \infty$ for some $\mu > 0$.
Then also $\GI_d(\omega, \frac{2d}{d - 1}) < \infty$ by Lemma \ref{l:GIUpperBound}.
Take any $L > \GI_d(\omega, \frac{2d}{d - 1})$, and let $C > 0$ correspond to it.
Then, for any $t \geq 0$ there exists a $s > 0$ such that
	\[ s^{\frac{2}{d-1}} + \omega\left(\frac{t^d}{s}\right) \leq L \omega(t) + \frac{2d}{d - 1} C . \]
We have $s = t^{d - \sigma}$ for some $\sigma \in \R$ .
By possibly enlarging $L$ and $C$, we may assume $1 \leq \sigma \leq d$.
Hence, for each $t$ there exists a $1 \leq \sigma \leq d$ for which
	\[ t^{2 \frac{d - \sigma}{d-1}} + \omega(t^\sigma) \leq L \omega(t) + \frac{2d}{d - 1} C . \]
We conclude \eqref{eq:GaussianInterpolation} holds, namely, $\omega$ satisfies quadratic interpolation of dimension $d$.
\end{proof}

We now establish the following equivalencies for quadratic interpolation.
These will play a crucial role in what follows.

\begin{lemma}
\label{l:GIWeightSeq}
Let $\mu, r > 0$.
\begin{itemize}
\item[$(i)$] If $\GI_d(\omega, \mu) < \mu / r$, then,
	\begin{equation}
		\label{eq:GIWeightSeq}
		e^{\frac{1}{r} \varphi^*(r n)} \lesssim (\sqrt{n!})^{1 - \frac{1}{d}} e^{\frac{1}{\mu} \varphi^*\left(\frac{\mu n}{d}\right)} . 
	\end{equation}

\item[$(ii)$] If \eqref{eq:GIWeightSeq} holds for $\mu, r$, then, $\GI_d(\omega, \mu) \leq \mu / r$
\end{itemize}
\end{lemma}

Before moving on to the proof, we remark that, if $\log t = o(\omega(t))$, then \eqref{eq:GIWeightSeq} can only hold if $r < \mu$, since $\sqrt{n!} = o(e^{\frac{1}{\varepsilon} \varphi^*(\varepsilon n)})$ for any $\varepsilon > 0$. 

\begin{proof}
When $\omega \asymp \log_+$, it satisfies quadratic interpolation and \eqref{eq:GIWeightSeq} is true for $\mu = d r$.
As $\GI_d(\omega, \mu) = d$ (see also Corollary \ref{c:GILog}), $(i)$ and $(ii)$ become clear.
We may now assume that $\log t = o(\omega(t))$.

Consider, for $\mu > 0$, the function
	\[ \psi_\mu(u) = \inf_{v \geq 0} \left[ \frac{d - 1}{2d} e^{v \frac{2d}{d - 1}} + \frac{1}{\mu} \varphi(d(u - v)) \right]  , \qquad u \geq 0 . \]
Note that $u = o(\psi_\mu(u))$.
Also, by considering $v = \frac{d - 1}{d} u$, we have
	\[ \psi_\mu(u) \leq \frac{d-1}{2d} e^{2 u} + \frac{1}{\mu} \varphi(u) \lesssim e^{2u} . \]
This shows that, for $u$ large enough, the minimum in $\psi_\mu(u)$ is reached at $v^* \leq u$.
By the convexity of $\varphi$ on $[0, \infty)$, we obtain that $\psi_\mu$ is eventually convex.

Set
	\[ M_{\mu}(v) = \sup_{u \geq 0} [u v - \psi_{\mu}(u)] , \qquad v \geq 0 . \]
Note that as $v \to \infty$, the point $u^*$ where the maximum of the defining expression for $M_\mu(v)$ is reached also diverges to $\infty$.
We show
	\begin{equation}
		\label{eq:MExplicit} 
		\left| M_\mu(v) - \frac{d - 1}{2d} v \log\left(\frac{v}{e}\right) - \frac{1}{\mu} \varphi^*\left(\frac{\mu v}{d}\right) \right| \leq C , 
	\end{equation}
for some $C > 0$.	
Indeed, for any $u$ large enough, say $u \geq u_0$, $\psi_\mu(u)$ reaches its minimum at $v_0^* \leq u$.
Put $v^*_1 = u - v_0^*.$
Then, for $v$ large enough,
	\begin{align*}
		M_{\mu}(v) 
		&= \sup_{u \geq u_0} [(v_0^* + v_1^*) v - \psi_{\mu}(v_0^* + v_1^*)] \\
		&\leq \sup_{v_0^* \geq 0} [v_0^* v - \frac{d - 1}{2d} e^{v_0^* \frac{2d}{d - 1}}] + \sup_{v_1^* \geq 0} [v_1^* v - \frac{1}{\mu} \varphi(d v_1^*)] \\
		&= \frac{d - 1}{2d} v \log\left(\frac{v}{e}\right) + \frac{1}{\mu} \varphi^*\left(\frac{\mu v}{d}\right) .
	\end{align*}
On the other hand,
	\begin{align*}
		M_{\mu}(v) 
		&= \sup_{u_0, u_1 \geq 0} [(u_0 + u_1) v - \psi_{\mu}(u_0 + u_1)] \\
		&\geq \sup_{u_0 \geq 0} [u_0 v - \frac{d - 1}{2d} e^{u_0 \frac{2d}{d - 1}}] + \sup_{u_1 \geq 0} [u_1 v - \frac{1}{\mu} \varphi(d u_1)] \\
		&= \frac{d - 1}{2d} v \log\left(\frac{v}{e}\right) + \frac{1}{\mu} \varphi^*\left(\frac{\mu v}{d}\right) .
	\end{align*}
Hence, \eqref{eq:MExplicit} has been established.
Stirling's approximation \eqref{eq:Stirling} yields
	\begin{equation} 
		\label{eq:UpperBoundM}
		e^{M_{\mu}(n)} \lesssim (\sqrt{n!})^{1 - \frac{1}{d}} e^{\frac{1}{\mu} \varphi^*\left(\frac{\mu n}{d}\right)} , \quad n \in \N , 
	\end{equation}
and, for each $\kappa > 1$,
	\begin{equation}
		\label{eq:LowerBoundM}
		(\sqrt{n!})^{1 - \frac{1}{d}} e^{\frac{1}{\mu} \varphi^*\left(\frac{\mu n}{d}\right)} \lesssim \kappa^{n} e^{M_{\mu}(n)} , \qquad n \in \N . 
	\end{equation}
We now proceed to show the two statements separately.

$(i)$ One verifies the existence of $C_0> 0$ such that, for any $u \geq 0$,
	\[
		\inf_{s > 0} \left[ \frac{d - 1}{2d} s^{\frac{2}{d - 1}} + \frac{1}{\mu} \omega\left(\frac{e^{du}}{s}\right)\right]
		\geq \inf_{s \geq 1} \left[ \frac{d - 1}{2d} s^{\frac{2}{d - 1}} + \frac{1}{\mu} \omega\left(\frac{e^{du}}{s}\right)\right] - C_0
		= \psi_\mu(u) - C_0 .
	\]
Suppose $\mu / r > \GI_d(\omega, \mu)$. 
Then, by \eqref{eq:UpperBoundM},
	\[
		e^{\frac{1}{r} \varphi^*\left(r n\right)} 
		= e^{\sup_{u \geq 0} \left[ u n - \frac{1}{r} \varphi(u) \right]} 
		\lesssim e^{\sup_{u \geq 0} [u n - \psi_\mu(u)]} 
		= e^{M_\mu(n)}
		\lesssim (\sqrt{n!})^{1 - \frac{1}{d}} e^{\frac{1}{\mu} \varphi^*\left(\frac{\mu n}{d}\right)}  .
	\]

$(ii)$ Suppose \eqref{eq:GIWeightSeq} holds for certain $\mu, r > 0$.
By \eqref{eq:alpha->0}, Lemma \ref{l:AddingSmallRegularity}, and \eqref{eq:LowerBoundM}, we can find for any $s' < r$ some $C_1 > 0$ such that
	\[ e^{\frac{1}{s'} \varphi^*(s' n)} \leq C_1 e^{M_\mu(n)} , \qquad n \in \N . \]
Take any other $s < s'$. For arbitrary $v \geq 0$, we have $n \leq v < n + 1$ for some $n \in \N$.
We get, by the convexity of $\varphi^*$,
	\begin{align*} 
		e^{\frac{1}{s} \varphi^*(s v)} 
		&\leq e^{\frac{1}{s} \varphi^*(s(n+1))} 
		= e^{\frac{1}{s} \varphi^*\left(\frac{s}{s'} (s' n) + \frac{s' - s}{s'} \left(\frac{s' s}{s' - s}\right)\right)}
		\leq e^{\frac{1}{s'} \varphi^*(s' n)} e^{\frac{s' - s}{s' s} \varphi^*\left(\frac{s' s}{s' - s}\right)} \\
		&\leq C_1 e^{\frac{s' - s}{s' s} \varphi^*\left(\frac{s' s}{s' - s}\right)} e^{M_{\mu}(n)} 
		\leq C_1 e^{\frac{s' - s}{s' s} \varphi^*\left(\frac{s' s}{s' - s}\right)}  e^{M_\mu(v)} . 
	\end{align*}
Hence,
	\[ \frac{1}{s} \varphi^*(sv) \leq M_\mu(v) + C_2 , \qquad v \geq 0 , \]
for some $C_2 > 0$. 
Now, as $\psi_\mu$ is eventually convex, there is some $C_3 > 0$ such that
	\[ \psi_\mu(u) \leq \sup_{v \geq 0} [u v - M_\mu(v)] + C_3 , \]
and thus
	\[ \psi_\mu(u) \leq \sup_{v \geq 0} [uv - \frac{1}{s} \varphi^*(sv)] + C_2 + C_3 = \frac{1}{s} \varphi(u) + C_2 + C_3 . \]
From here we may conclude that $\GI_d(\omega, \mu) < \mu / s$.
\end{proof}

We now provide more tangible upper and lower bounds for the indices $\GI_d(\omega, \mu)$.

\begin{lemma}
\label{l:SuffCondGI}
~
\begin{itemize}
\item[$(i)$] For any $\mu > 0$,
	\begin{equation}
		\label{eq:SuffCondGI}
		\GI_d(\omega, \mu) \leq \limsup_{t \to \infty} \frac{\omega(t^d)}{\omega(t)} .
	\end{equation}

\item[$(ii)$] Suppose, for certain $1 < \sigma_0 < \sigma_1 < d$, that $\omega(t) = o\left(t^{\frac{2}{\sigma_1} \frac{d-\sigma_0}{d-1}}\right)$.
Then, for any $\mu > 0$,
	\begin{equation}
		\label{eq:GIdLowerBound} 
		\GI_d(\omega, \mu)  \geq \liminf_{t \to \infty} \frac{\omega(t^{\sigma_0})}{\omega(t)}. 
	\end{equation}
\end{itemize}
\end{lemma}

\begin{proof}
$(i)$ 
Suppose $\omega(t^d) \leq L \omega(t) + C$ for certain $L, C > 0$.
For any $\mu > 0$ and $t \geq 0$, we then have
	\[ 
		\inf_{s \geq 0} \left[ \frac{d - 1}{2d} s^{\frac{2}{d - 1}} + \frac{1}{\mu} \omega\left( \frac{t^d}{s} \right) \right] 
		\leq 	\frac{d - 1}{2d} + \frac{1}{\mu} \omega\left( t^d \right)
		\leq \frac{L}{\mu} \omega(t) + \frac{C}{\mu} + \frac{d-1}{2d} .
	\]
	
$(ii)$ 
Suppose $L \omega(t) \leq \omega(t^{\sigma_0}) + C$ for certain $L, C > 0$.
We have $\omega(t^{\sigma_1}) = o\left(t^{2 \frac{d - \sigma_0}{d - 1}}\right)$.
Let $0 < \varepsilon < 1$.
For $t$ large enough,
	\begin{align*}
		\inf_{s \geq 0} \left[ \frac{d - 1}{2d} s^{\frac{2}{d - 1}} + \frac{1}{\mu} \omega\left( \frac{t^d}{s} \right) \right] 
		&\geq \inf_{\sigma > \sigma_0} \left[ \frac{d-1}{2d} t^{2 \frac{d - \sigma}{d - 1}} + \frac{1}{\mu} \omega(t^{\sigma}) \right] \\
		&\geq \frac{1 - \varepsilon}{\mu} \omega(t^{\sigma_0}) 
		\geq (1 - \varepsilon) \left[\frac{L}{\mu} \omega(t) - \frac{C}{\mu}\right] .
	\end{align*}
\end{proof}

\begin{corollary}
\label{c:GILog}
For any $a \geq 0$ the weight function $(\log_+ t)^{1 + a}$ satisfies quadratic interpolation of dimension $d$ and we have $\GI_d((\log_+ t)^{1 + a}, \mu) = d^{1 + a}$ for any $\mu > 0$.
\end{corollary}

\begin{proof}
Put $\omega(t) = (\log_+ t)^{1 + a}$.
For any $\sigma > 0$ we have
	\[ \lim_{t \to \infty} \frac{\omega(t^\sigma)}{\omega(t)} = \sigma^{1+a} . \]
Lemma \ref{l:SuffCondGI} yields $\GI_d(\omega, \mu) = d^{1 + a}$ for any $\mu > 0$.
\end{proof}

\begin{corollary}
\label{c:t^aNotQuadraticInterpol}
For any $0 < a < 2$, the weight function $t^a$ does not satisfy quadratic interpolation of dimension $d$.
\end{corollary}

\begin{proof}
Put $\omega(t) = t^a$.
Since $a < 2$ there exist $1 < \sigma_0 < \sigma_1 < d$ such that $t^a = o\left(t^{\frac{2}{\sigma_1} \frac{d-\sigma_0}{d-1}}\right)$.
For any $R \geq 1$ and $t \geq R$, we have 
	\[ R^{a (\sigma_0 - 1)} \omega(t) = (R^{\sigma_0-1} t)^{a} \leq (t^{\sigma_0})^a = \omega(t^{\sigma_0}) . \]
Therefore $\GI_d(\omega, \mu) \geq R^{a (\sigma_0 - 1)}$ for any $\mu > 0$ by Lemma \ref{l:SuffCondGI}$(ii)$.
As $R$ was arbitrary, we conclude that $\GI_d(\omega, \mu) = \infty$ for any $\mu > 0$.
By Lemma \ref{l:GIViaCoeff}, $\omega$ cannot satisfy quadratic interpolation of dimension $d$.
\end{proof}

\subsection{Proofs of the Theorems}
\label{sec:Counterexample}

We now prove each theorem stated in the introduction of this section.
For this, we will extensively use Lemma \ref{l:GIWeightSeq}.
In fact, Theorem \ref{t:CoordinateWiseSucces} follows directly from it and Proposition \ref{p:PreciseEstGenCaseCoordinates}.

\begin{proof}[Proof of Theorem \ref{t:CoordinateWiseSucces}]
Suppose $f \in L^2(\R^d)$ satisfies \eqref{eq:TFDecayCoordinates} for some $\lambda > 0$.
For any $\mu < (\mathfrak{C}_{2, d}(\omega) \lambda)^{-1}$, we have, by Proposition \ref{p:PreciseEstGenCaseCoordinates},
	\begin{equation}
		\label{eq:CoordinateWiseIntermediate} 
		|H_A(f, \alpha)| \lesssim (\alpha!)^{\frac{1}{2d}} e^{-\frac{1}{\mu} \varphi^*\left(\frac{\mu |\alpha|}{d}\right)}. 
	\end{equation}
It follows from Lemma \ref{l:GIViaCoeff} that $\GI_d(\omega, \mu) < \infty$; therefore, Lemmas \ref{l:AddingSmallRegularity} and \ref{l:GIWeightSeq}$(i)$ show that, for any $r < \mu / [\ModKOm{d^{\frac{d-1}{2d}}} \GI_d(\omega, \mu)]$,
	\[  (\alpha!)^{\frac{1}{2d}} e^{-\frac{1}{\mu} \varphi^*\left(\frac{\mu |\alpha|}{d}\right)} \lesssim \sqrt{\alpha!} e^{-\frac{1}{r} \varphi^*(r |\alpha|)} . \]
Combining this with \eqref{eq:CoordinateWiseIntermediate} gives the result.	
\end{proof}

Applying Theorem \ref{t:CoordinateWiseSucces} to the weight functions $(\log_+ t)^{1+a}$, $a \geq 0$, yields us Theorems \ref{t:LogCaseGeneral} and \ref{t:CoordinatesLog}.

\begin{proof}[Proof of Theorem \ref{t:LogCaseGeneral}]
The cases $(i)$ and $(iii)$ follow immediately from Theorem \ref{t:AandBNotInverses}.
We now suppose $B = A^{-1}$.
By Corollary \ref{c:GILog} we have that $\GI_d(\log_+ t, \mu) = d$ for any $\mu > 0$.
Then, also using Lemma \ref{l:ExamplesAlpha}$(ii)$, we see that $\mathfrak{C}_{3, d}(\log_+ t, \lambda) = d$ for any $\lambda > 0$.
Suppose $f \in L^2(\R^d)$ satisfies \eqref{eq:PolynomialDecayCoordinates} for some $N \in \N$.
By Theorem \ref{t:CoordinateWiseSucces}, we have that
	\[ |H_A(f, \alpha)| \lesssim_r \sqrt{\alpha!} e^{-\frac{1}{r} \varphi^*(r |\alpha|)} , \qquad 0 < r < \frac{1}{dN} . \]
Now, $e^{-\frac{1}{r} \varphi^*(r |\alpha|)} \neq 0$ if and only if $r |\alpha| \leq 1$.
Therefore $H_A(f, \alpha)$ can only be non-zero if $|\alpha| \leq dN$.
Then $f(x) = p(x) e^{-\frac12 \ev{Ax}{x}}$ with $p(x) = \sum_{|\alpha| \leq dN} H_A(f, \alpha) h_{A, \alpha}(x) e^{\frac12 \ev{Ax}{x}}$.
In particular, $p$ has degree at most $d N$.
From \eqref{eq:PolynomialDecayCoordinates}, it then follows that $p$ has degree at most $N$ in each variable.
\end{proof}

\begin{proof}[Proof of Theorem \ref{t:CoordinatesLog}]
By Lemmas \ref{l:ExamplesAlpha}$(ii)$ and \ref{c:GILog} we have $\mathfrak{C}_{3,d}((\log_+ t)^{1 + a}, \lambda) = d^{1+a}$.
The result follows from Theorem \ref{t:CoordinateWiseSucces} using \eqref{eq:YoungSeq-log^1+a}.
\end{proof}

We now consider the case where quadratic interpolation is not satisfied.
We start by verifying that specific bounds on the skewed Bargmann transform imply the coordinate-wise time-frequency decay \eqref{eq:TFDecayCoordinates}, similarly to Lemma \ref{l:BargmannToTFDecay}.
This will allow us to construct our examples in the Fock space, that is, as entire functions.

\begin{lemma}
\label{l:BargmannToTFDecayCoordinates}
Let $\omega$ be a weight function satisfying $(\alpha)$ such that $\omega(t) = o(t^2)$.
Let $F$ be an entire function on $\C^d$ satisfying
	\[ |F(z)| \lesssim \exp [ \eta \omega(|z_j|) + \frac{1}{4} (|z|^2 - |z_j|^2) ] , \qquad j \in \{ 1, \ldots, d \} , \]
for some $\eta > 0$ and $\varepsilon > 0$.
Then, $F = \B_A f$ for some $f \in L^2(\R^d)$ for which \eqref{eq:TFDecayCoordinates} holds for $\lambda > \ModOm \ModKOm{\sqrt{2}} \eta$.
\end{lemma}

\begin{proof}
There exists $f \in L^2(\R^d)$ such that $F = \B_A f$ and
	\[ \left| f\left(\frac{x}{2}\right) \right| \lesssim e^{\frac{1}{8} \ev{Ax}{x}} \int_{\R^d} |V_{\phi_A} f(x, \xi)| d\xi .  \]
Let $L > \ModOm$ and $L' > \ModKOm{\sqrt{2}}$.
By \eqref{eq:BargmannSTFT} we find, for $j \in \{ 1, \ldots, d \}$,
	\begin{multline*} 
		|V_{\phi_A} f(x, \xi)| 
		\lesssim
		\exp\left[ -\frac14 |\ev{Ax}{v_j}|^2 - \frac18 \sum_{k \neq j} |\ev{Ax}{v_j}|^2 + L \lambda \omega\left(\frac{|\ev{Ax}{v_j}|}{\sqrt{2}}\right)  \right.  \\
		\left. 
		- \frac{1}{4} |\ev{A^{-1} \xi}{w_j}|^2 + L \lambda \omega\left(\frac{|\ev{A^{-1} \xi}{w_j}|}{\sqrt{2}}\right) - \frac{1}{8} \sum_{k \neq j} |\ev{A^{-1} \xi}{w_k}|^2 \right] .  
	\end{multline*}
Integrating with respect to $\xi$ gives, for $1 \leq j \leq d$,
	\begin{align*}
		|f(x)| &\lesssim e^{\frac{1}{2} \ev{Ax}{x} - |\ev{Ax}{v_j}|^2 - \frac12 \sum_{k \neq j} |\ev{Ax}{v_j}|^2 + \lambda L \omega(\sqrt{2} |\ev{A x}{v_j}|)} \\
		&\lesssim e^{-\frac{1}{2} |\ev{A x}{v_j}|^2 + L L' \lambda \omega(|\ev{A x}{v_j}|)} .
	\end{align*}
The other bounds follow similarly by \eqref{eq:STFTFourier}.
\end{proof}

We first consider part $(i)$ of Theorem \ref{t:CounterExample}.
We need the following technical result.

\begin{lemma}
\label{l:NotGI}
Let $d \geq 2$ and $\omega$ be a weight function satisfying $(\alpha)$, $(\delta)$, and $\omega(t) = o(t^2)$, but not quadratic interpolation of dimension $d$.
Then, for every $r, \varepsilon > 0$,
	\begin{equation} 
		\label{eq:NotGI}
		\sup_{n \in \N} \frac{\left( (dn)! \right)^{\frac{1}{2d}} e^{-\frac{1}{r} \varphi^*(r n)}}{\sqrt{(dn)!} e^{-\frac{1}{\varepsilon} \varphi^*(\varepsilon d n)}} = \infty . 
	\end{equation}
\end{lemma}

\begin{proof}
Suppose \eqref{eq:NotGI} does not hold for certain $r, \varepsilon >0$.
Then, there exists a $C > 0$ such that
	\[ \left( (dn)! \right)^{\frac{1}{2d}} e^{-\frac{1}{r} \varphi^*(r n)} \leq C \sqrt{(dn)!} e^{-\frac{1}{\varepsilon} \varphi^*(\varepsilon d n)} .   \]
Take any $m \geq d$, then $dn \leq m < d(n + 1)$ for some $n \geq 1$.
Using the convexity of $\varphi$, we find
	\begin{align*}
		\left( m! \right)^{\frac{1}{2d}} e^{-\frac{1}{r} \varphi^*\left(\frac{r m}{d}\right)}
		&\leq \left( (d(n+1))! \right)^{\frac{1}{2d}} e^{-\frac{1}{r} \varphi^*(rn)}
		\leq 2^{\frac{n + 1}{2}} (d!)^{\frac{1}{2d}} \left( (dn)! \right)^{\frac{1}{2d}} e^{-\frac{1}{r} \varphi^*(r n)} \\
		&\leq C 2^{\frac{n + 1}{2}} (d!)^{\frac{1}{2d}} [\sqrt{(dn)!} e^{-\frac{1}{\varepsilon} \varphi^*(\varepsilon d n)}] \\
		&\leq C 2^{\frac{n + 1}{2}} (d!)^{\frac{1}{2d}} e^{\frac{1}{\varepsilon} \varphi^*(\varepsilon d)} [\sqrt{(dn)!} e^{-\frac{2}{\varepsilon} \varphi^*\left(\frac{\varepsilon d(n + 1)}{2}\right)}] \\
		&\leq C (d!)^{\frac{1}{2d}} e^{\frac{1}{\varepsilon} \varphi^*(\varepsilon d)} [\sqrt{m!} 2^{\frac{m}{2}} e^{-\frac{2}{\varepsilon} \varphi^*\left(\frac{\varepsilon m}{2}\right)}] .
	\end{align*}
Let $\varepsilon' < \varepsilon / [2 \ModKOm{\sqrt{2}}]$. 
By Lemma \ref{l:AddingSmallRegularity},
	\[ e^{\frac{1}{\varepsilon'} \varphi^*(\varepsilon' m)} \lesssim (\sqrt{m!})^{1 - \frac{1}{d}} e^{\frac{1}{r} \varphi^*\left(\frac{r m}{d}\right)} , \qquad m \in \N .  \]
Lemma \ref{l:GIWeightSeq}$(ii)$ now gives $\GI_d(\omega, r) < \infty$.
This contradicts Lemma \ref{l:GIViaCoeff}.
\end{proof}

\begin{proof}[Proof of Theorem \ref{t:CounterExample}$(i)$]
Fix $\lambda > 0$.
Take any $\eta < \lambda / [\ModOm \ModKOm{\sqrt{2}}]$.
We define the entire function $F$ on $\C^d$ as
	\[ F(z) = \sum_{n = 0}^\infty (2)^{-\frac{nd}{2}} \frac{z_1^n \cdots z_d^n}{(n!)^{\frac{d-1}{2}} e^{\eta \varphi^*\left(\frac{n}{\eta}\right)}} . \]
For any $j \in \{1, \ldots, d\}$,
	\begin{align*} 
		|F(z)| 
		&\leq \left( \sup_{n \in \N} |z_j|^n e^{- \eta \varphi^*\left(\frac{n}{\eta}\right)} \right) \prod_{k \neq j} \left( \sup_{n \in \N} \frac{2^{-\frac{n}{2}} |z_k|^n}{\sqrt{n!}} \right) \cdot \sum_{n = 0}^\infty 2^{-\frac{n}{2}} \\
		&\lesssim \exp\left[ \eta \omega(|z_j|) + \frac{1}{4} (|z|^2 - |z_j|^2) \right] .
	\end{align*}
Hence, by Lemma \ref{l:BargmannToTFDecayCoordinates}, we have $F = \B_d f$ for some $f \in L^2(\R^d)$ satisfying \eqref{eq:TFDecayCoordinates} for $\lambda$.

Let $r > [ \ModKOm{2^{\frac{d}{2}} d^{\frac{1}{2d}}} \eta]^{-1}$.
In view of \eqref{eq:HermiteCoeffViaBargmann}, using Lemma \ref{l:AddingSmallRegularity}, we have
	\begin{align*}
	 	H_A(f, (n, \ldots, n)) 
		&= 2^{-\frac{nd}{2}} \sqrt{n!} e^{- \eta \varphi^*\left(\frac{n}{\eta}\right)} 
		\geq ((dn)!)^{\frac{1}{2d}} (2^{\frac{d}{2}} d^{\frac{1}{d}})^{-n} e^{- \eta \varphi^*\left(\frac{n}{\eta}\right)} \\
		&\gtrsim \left( (dn)! \right)^{\frac{1}{2d}} e^{-\frac{1}{r} \varphi^*(r n)}  .
	\end{align*}
Then, by Lemma \ref{l:NotGI}, for any $\varepsilon > 0$,
	\[ 
		\sup_{\alpha \in \N^d} \frac{|H_A(f, \alpha)|}{\sqrt{\alpha!} e^{-\frac{1}{\varepsilon} \varphi^*(\varepsilon |\alpha|)}} 
		\geq \sup_{n \in \N} \frac{|H_A(f, (n, \ldots, n))|}{\sqrt{(dn)!} e^{-\frac{1}{\varepsilon} \varphi^*(\varepsilon d n)}} 
		\gtrsim \sup_{n \in \N} \frac{\left( (dn)! \right)^{\frac{1}{2d}} e^{-\frac{1}{r} \varphi^*(r n)}}{\sqrt{(dn)!} e^{-\frac{1}{\varepsilon} \varphi^*(\varepsilon d n)}}
		= \infty .
	\]
Therefore, $f$ satisfies all conditions.
\end{proof}

Let us now consider part $(ii)$ of Theorem \ref{t:CounterExample}.
This time, we need the following result.

\begin{lemma}
\label{l:littleohdimensionalYoungConjugateSequence}
Let $d \geq 2$ and $\omega$ be a weight function satisfying $(\alpha)$, $(\delta)$, and $\omega(t) = o(t^2)$, but not quadratic interpolation of dimension $d$.
There exists a non-decreasing sequence $(r_n)_{n \in \N}$ of positive numbers going to $\infty$ such that, for every $\varepsilon > 0$,
	\begin{equation}
		\label{eq:littleohdimensionalYoungConjugateSequence} 
		\sup_{n \in \N} \frac{\left( (dn)! \right)^{\frac{1}{2d}} e^{-\frac{1}{r_n} \varphi^*(r_n n)}}{\sqrt{(dn)!} e^{-\frac{1}{\varepsilon} \varphi^*(\varepsilon d n)}} = \infty .
	\end{equation}
\end{lemma}

\begin{proof}
Using Lemma \ref{l:NotGI}, we find for any $k \geq 1$ there is some $n_k \in \N$ such that
	\[ \frac{\left((dn_k)! \right)^{\frac{1}{2d}} e^{-\frac{1}{k} \varphi^*(k n_k)}}{\sqrt{(dn_k)!} e^{- k \varphi^*\left(\frac{1}{k} d n_k\right)}} \geq k . \]
Moreover, we may assume $(n_k)_{k \in \N}$ is a strictly increasing sequence.
Then, we put $r_{0} = \cdots = r_{n_1 - 1} = 1$ and $r_n = k$ whenever $n_k \leq n < n_{k+1}$.
Hence, $(r_n)_{n \in \N}$ is a non-decreasing sequence of positive numbers tending to $\infty$. 
For any $\varepsilon > 0$, let $k_\varepsilon = \lceil \varepsilon^{-1} \rceil$.
We see that
	\[ 
		\frac{\left((dn_k)! \right)^{\frac{1}{2d}} e^{-\frac{1}{r_{n_k}} \varphi^*(r_{n_k} n_k)}}{\sqrt{(dn_k)!} e^{-\frac{1}{\varepsilon} \varphi^*(\varepsilon d n_k)}}
		\geq 	\frac{\left((dn_k)! \right)^{\frac{1}{2d}} e^{-\frac{1}{r_{n_k}} \varphi^*(r_{n_k} n_k)}}{\sqrt{(dn_k)!} e^{- k \varphi^*\left(\frac{1}{k} d n_k\right)}}
		\geq k , 
		\qquad \forall k \geq k_\varepsilon . 
	\]
This concludes our proof.
\end{proof}

\begin{proof}[Proof of Theorem \ref{t:CounterExample}$(ii)$]
Let $(r_n)_{n \in \N}$ be a non-decreasing sequence as in Lemma \ref{l:littleohdimensionalYoungConjugateSequence}.
Take any $L > \ModKOm{2^{\frac{d}{2}} d^{\frac{1}{2d}}}$ and put $s_n = r_n / L$ for $n \in \N$.
By Lemma \ref{l:AddingSmallRegularity}, there is a $C > 0$ for which
	\[ (2^{\frac{d}{2}} d^{\frac{1}{2d}})^n e^{\frac{1}{s_n} \varphi^*(s_n n)} \leq C e^{\frac{1}{r_n} \varphi^*(r_n n)} , \qquad n \in \N . \]

We define the entire function $F$ on $\C^d$ as
	\[ F(z) = C \sum_{n = 0}^\infty 2^{-\frac{nd}{2}} \frac{z_1^n \cdots z_d^n}{(n!)^{\frac{d-1}{2}} e^{\frac{1}{s_n} \varphi^*(s_n n)}} . \]
Then
	\[ |F(z)| \lesssim \exp \left[ \frac{1}{s_n} \omega(|z_j|) + \frac{1}{4} (|z|^2 - |z_j|^2) \right] , \qquad j \in \{1, \ldots, d\} . \]
As $r_n$ goes to $\infty$, and thus also $s_n$, Lemma \ref{l:BargmannToTFDecayCoordinates} shows that $F = \B_A f$ for some $f \in L^2(\R^d)$ satisfying \eqref{eq:TFDecayCoordinates} for every $\lambda > 0$.
By \eqref{eq:HermiteCoeffViaBargmann},
	\[ H_A(f, (n, \ldots, n)) = C 2^{-\frac{nd}{2}} \sqrt{n!} e^{-\frac{1}{s_n} \varphi^*(s_n n)} \gtrsim \left((dn)!\right)^{\frac{1}{2d}} e^{-\frac{1}{r_n} \varphi^*(r_n n)}   . \]
Then, for any $\varepsilon > 0$,
	\[ 
		\sup_{\alpha \in \N^d} \frac{|H_A(f, \alpha)|}{\sqrt{\alpha!} e^{-\frac{1}{\varepsilon} \varphi^*(\varepsilon |\alpha|)}} 
		\geq \sup_{n \in \N} \frac{|H_A(f, (n, \ldots, n))|}{\sqrt{(dn)!} e^{-\frac{1}{\varepsilon} \varphi^*(\varepsilon d n)}} 
		\geq \sup_{n \in \N} \frac{\left((dn)!\right)^{\frac{1}{2d}} e^{-\frac{1}{r_n} \varphi^*(r_n n)}}{\sqrt{(dn)!} e^{-\frac{1}{\varepsilon} \varphi^*(\varepsilon d n)}}
		= \infty .
	\]
\end{proof}

We end with the proof of Theorem \ref{t:Countexamplet^a}.

\begin{proof}[Proof of Theorem \ref{t:Countexamplet^a}]
Part $(i)$ follows from Proposition \ref{p:PreciseEstGenCaseCoordinates} by using \eqref{eq:CompSqrtFactAndYoung} and Lemma \ref{l:ExamplesAlpha}$(i)$, while $(ii)$ follows directly from Theorem \ref{t:CounterExample}$(i)$.

Let us now show $(iii)$.
Define the entire function
	\[ F(z) = \sum_{n = 0}^\infty \frac{1}{(\log_+ n)^n} \frac{z_1^n \cdots z_d^n}{(n!)^{\frac{d-1}{2}} (n!)^{\frac{1}{a}}} . \]
As in the proof of Theorem \ref{t:CounterExample}$(ii)$, one sees that $F = \B_A f$ for some $f \in L^2(\R^d)$ satisfying \eqref{eq:Coordinatewiset^a} for any $\lambda > 0$.
But, for any $n \in \N$, $h > 0$, and $\varepsilon > 0$,
	\[ 
		H_A(f, (n, \ldots, n)) 
		= \frac{(n!)^{\frac12 - \frac{1}{a}}}{(\log_+ n)^n} 
		\gtrsim \left(\frac{h^d}{d^{ \frac12 - \frac{1}{a} - \varepsilon}}\right)^n (n!)^{\frac12 - \frac{1}{a} - \varepsilon} 
		\geq h^{dn} ((dn)!)^{\frac{1}{d} \left( \frac12 - \frac{1}{a} - \varepsilon \right)}
	. \]
As $\varepsilon > 0$ was chosen arbitrarily, we have shown \eqref{eq:t^aFailOtherPowers} holds for $f$.
\end{proof}

\end{document}